\documentclass[reqno]{amsart}
\usepackage[left=1in,right=1in,top=1in,bottom=1in]{geometry}
\usepackage{tikz}
\usetikzlibrary{shapes,decorations,calc,arrows}
\usepackage[foot]{amsaddr}
\usepackage{amsthm}
\usepackage{amscd}
\usepackage{amsfonts}
\usepackage{amsmath}
\usepackage{amssymb}
\usepackage{mathrsfs}
\usepackage{multirow}
\usepackage{verbatim}
\usepackage{url}
\usepackage[hidelinks]{hyperref}
\usepackage{graphicx}
\usepackage{cite}
\usepackage{fancyhdr}
\usepackage{yfonts}
\usepackage{setspace}
\usepackage{titlesec}
\usepackage{enumitem}

\usepackage{ dsfont }

\usepackage{tocloft} 

\titleformat{\section}[hang]
{\normalfont\Large\bfseries}
{\thesection.}{0.5em}{}

\titlespacing*{\section}{0pc}{2pc}{0.25pc}

\titleformat{\subsection}[runin]
{\normalfont\large\bfseries}
{\thesubsection}{0.5em}{}

\titlespacing{\subsection}{0pc}{1.5pc}{0.5pc}

\newcommand{\N}{\mathbb{N}}

\newcommand{\Q}{\mathbb{Q}}
\newcommand{\R}{\mathbb{R}}
\newcommand{\C}{\mathbb{C}}
\renewcommand{\H}{\mathcal{H}}

\newcommand{\J}{\mathscr{J}}

\newcommand{\HS}{\text{HS}}
\newcommand{\1}{\mathds{1}}

\newcommand{\Tr}{\text{Tr}}
\newcommand{\tr}{\text{tr}}

\newcommand{\abs}[1]{\left|#1\right|}
\newcommand{\norm}[1]{\left\|#1\right\|}
\newcommand{\set}[1]{\left\{#1\right\}}
\newcommand{\paren}[1]{\left(#1\right)}
\newcommand{\ang}[1]{\left\langle#1\right\rangle}

\newcommand{\Der}{\mathsf{Der}}
\newcommand{\InnDer}{\mathsf{InnDer}}

\DeclareMathOperator{\ev}{ev}

\newcommand{\II}{\rm{II}}

\newcommand{\<}{\left\langle}
\renewcommand{\>}{\right\rangle}

\newcommand{\dom}{\text{dom}}

\newcommand{\mc}[1]{\mathcal{#1}}

\newcommand{\mmop}{{M\bar\otimes M^\circ}}
\newcommand{\ttop}{{\tau\otimes\tau^\circ}}
\newcommand{\aaop}{{A\otimes A^\circ}}
\newcommand{\btxbtxop}{B\ang{T_X}\otimes B\ang{T_X}^\circ}

\definecolor{ggreen}{HTML}{7FDD99}

\definecolor{obnoxious}{HTML}{662288}

\newtheorem{thm}{Theorem}[section]
\newtheorem{prop}[thm]{Proposition}
\newtheorem{lem}[thm]{Lemma}
\newtheorem*{lem*}{Lemma}
\newtheorem{cor}[thm]{Corollary}

\theoremstyle{definition}
\newtheorem{defi}[thm]{Definition}
\newtheorem{ex}[thm]{Example}

\newtheorem{rem}[thm]{Remark}

\title{\textbf{On Free Stein Dimension}}
\author{Ian Charlesworth$^\circ$}
\address{$^\circ$Department of Mathematics, KU Leuven, Belgium \hfill \url{ian.charlesworth@kuleuven.be}}
\author{Brent Nelson$^\bullet$}
\address{$^\bullet$Department of Mathematics, Michigan State University \hfill \url{brent@math.msu.edu}}
\date{}

\begin{document}

\maketitle

\begin{abstract}
We establish several properties of the free Stein dimension, an invariant for finitely generated unital tracial $*$-algebras.
We give formulas for its behaviour under direct sums and tensor products with finite dimensional algebras
. Among a given set of generators, we show that (approximate) algebraic relations produce (non-approximate) bounds on the free Stein dimension.
Particular treatment is given to the case of separable abelian von Neumann algebras, where we show that free Stein dimension is a von Neumann algebra invariant.
In addition, we show that under mild assumptions $L^2$-rigidity implies free Stein dimension one.
Finally, we use limits superior/inferior to extend the free Stein dimension to a von Neumann algebra invariant---which is substantially more difficult to compute in general---and compute it in several cases of interest.
%
\end{abstract}



\section*{Introduction}

Given a tuple $X=(x_1,\ldots, x_n)$ of operators in a tracial von Neumann algebra $(M,\tau)$, the quantitative invariants associated to the distribution of $X$ with respect to $\tau$ are legion: the free entropy dimensions $\delta(X)$, $\delta_0(X)$, $\delta^*(X)$, and $\delta^\star(X)$; the free Fisher information $\Phi^*(X)$; the 1-bounded entropy $h(X)$; and so on. The behavior of these invariants often provides insights into the structure of the von Neumann algebra generated by this tuple. In \cite{CN21}, we defined new invariants called the \emph{free Stein irregularity} $\Sigma^*(X)$ and the \emph{free Stein dimension} $\sigma(X)$, whose primary advantage was being computable for a large number of examples of interest in free probability and operator algebras.

The notion of free Stein irregularity was motivated by the following observation. If $\J_X$ is the free Jacobian (see Section~\ref{subsec:derivation_spaces}), then $\Phi^*(X)<\infty$ if and only if $\1\in \dom(\J_X^*)$ where
    \[
        \1 := \left(\begin{array}{ccc} 1\otimes 1 & & 0 \\ &\ddots & \\ 0 & & 1\otimes 1 \end{array}\right),
    \]
and one views $\J_X\colon L^2(M,\tau)^n \to M_n(L^2(M\bar\otimes M^{\circ},\tau\otimes \tau^\circ)$ as a densely defined operator (see \cite[Definition 6.1]{VoiV}). The free Stein irregularity $\Sigma^*(X)$ is defined as the distance between $\1$ and $\dom(\J_X^*)$, and so $\Phi^*(X)<\infty$ implies $\Sigma^*(X)=0$. However, the converse is not true: it may be that $\1$ lies in $\overline{\dom(\J_X^*)}\setminus \dom(\J_X^*)$. Thus ``zero free Stein irregularity'' is a weakening of finite free Fisher information, and the implications of this weaker condition were explored in \cite{CN21}.

The free Stein dimension is given by $\sigma(X):=n - \Sigma^*(X)^2$, and was shown in \cite[Theorem 2.11]{CN21} to be $n$ times the von Neumann dimension of $\overline{\dom(\J_X^*)}$ as a right $M_n(M\bar\otimes M^\circ)$-module. It was also shown to be bounded above by the \emph{non-microstates} free entropy dimension $\delta^*(X)$ (see \cite[Corollary 4.4]{CN21}). Moreover, it was shown that $\sigma(X)$ is a unital $*$-algebra invariant: it depends only on $\tau$ and $\C\<X\>$---the unital $*$-algebra generated by $x_1,\ldots, x_n$---rather than the tuple $X$ itself (see \cite[Theorem 3.2]{CN21}). One might hope to push this result further and show the free Stein dimension is actually a \emph{von Neumann algebra} invariant, but the ambitious reader should note that this would resolve the free group factor isomorphism problem: if $X$ is a free semicircular family then $\Phi^*(X)<\infty$ and so $\Sigma^*(X)=0$, implying $\sigma(X)=n$.

In the present paper we refine the notion of the free Stein dimension and establish several new formulas with applications to yet more examples. In a shift from \cite{CN21}, we emphasize a perspective on the free Stein dimension that was essential for showing its algebraic invariance: $\overline{\dom(\J_X^*)}$ is isomorphic as a right module to a certain subspace of derivations on $\C\<X\>$ and valued in $L^2(M\bar\otimes M^\circ,\tau\otimes \tau^\circ)$ (see \cite[Lemma 3.1]{CN21}). Focusing on the unital $*$-algebra $\C\<X\>$ rather than the particular tuple $X$ has a number of algebraic advantages which we reap below.
In particular, we will work with the free Stein dimension of an algebra $A$ with a trace $\tau$ below---rather than that of a tuple of variables with a distribution---and so use $\sigma(A, \tau)$ rather than $\sigma(X)$, or more generally $\sigma(B\subset A,\tau)$ for derivations vanishing on a subalgebra $B$; see Definition~\ref{def:fsd}.

After some preliminaries in Section~\ref{sec:prelims}, we prove in Section~\ref{sec:direct_sum_and_amplification_formulae} that the free Stein dimension satisfies certain direct sum and tensor products formulas (see Theorem~\ref{thm:direct_sum_formula} and Corollary~\ref{cor:tensor_against_finite_dimensional_formula}). A key intermediate result is that one can always decompose $\sigma(A,\tau) = \sigma(B\subset A,\tau)+ \sigma(B,\tau)$ for any finite dimensional $*$-subalgebra $B\subset A$ (see Theorem~\ref{thm:excising_a_finite_dimensional_subalgebra}). This result also allows us to compute $\sigma(A,\tau)=t$ for a particular $A$ generating an interpolated free group factor $L(\mathbb{F}_t)$ (see Example~\ref{ex:free_graph_algebras}).

In Section~\ref{sec:algebraic_relations} we show how algebraic relations (even asymptotic ones) can be used to give effective upper bounds on the free Stein dimension. Notably, we show that the existence of a diffuse central element in $A$ implies $\sigma(A,\tau)=1$ (see Corollary~\ref{cor:diffuse_center}). We also show that in the case that the free Stein dimension is maximized (i.e. $\sigma(A,\tau)=n$ when $A$ contains a generating set $X=X^*$ of size $n$), adding certain certain elements of $A''$ to $A$ cannot increase the free Stein dimension (see Theorem~\ref{thm:maximal_fSd}).

In Section~\ref{sec:abelian_vNas}, we show that the free Stein dimension is an invariant for \emph{separable abelian} von Neumann algebras. In particular, if $M$ is a separable abelian von Neumann algebra with faithful normal trace $\tau$ and $P$ denotes its set of minimal projections, then
    \[
        \sigma(A,\tau)=1 - \sum_{p\in P} \tau(p)^2
    \]
for any finitely generated unital $*$-algebra $A\subset M$ satisfying $A''=M$ (see Theorem~\ref{thm:fSd_for_abelian_vNas}).

In Section~\ref{sec:L^2_rigidity}, using results from \cite{Pet09} and \cite{DI16} we show that if $A''$ is a non-amenable $\II_1$ factor and $A$ contains a non-amenability set, then $\sigma(A,\tau)>1$ implies $A''$ is \emph{not} $L^2$-rigid (see Theorem~\ref{thm:L^2-rigid}).

Finally, in Section~\ref{sec:upper_lower_fSd} we define new invariants for von Neumann algebras called the upper and lower free Stein dimension, which are defined via limits superior and inferior (respectively) of the net of free Stein dimensions of all finitely generated unital $*$-algebras, directed by inclusion. While these invariants are much more difficult to compute in practice, we show that they inherit many of the properties of the free Stein dimension and that they can computed in few a interesting cases (see Theorems~\ref{thm:fSd_is_one_if} and \ref{thm:extended_fSd_abelian_vNas}).

\subsection*{Acknowledgements}
We would like to thank Srivatsav Kunnawalkam Elayavalli, Michael Hartglass, Benjamin Hayes, Jesse Peterson, and Dimitri Shlyakhtenko for many helpful discussions related to this paper. The first and second authors were supported by NSF grants DMS-1803557 and DMS-1856683, respectively. The first author was also supported by long term structural funding in the form of a Methusalem grant from the Flemish Government.


\section{Preliminaries}\label{sec:prelims}

\subsection{Notation}
Throughout, a \emph{tracial von Neumann algebra} is a $(M,\tau)$ consisting of a finite von Neumann algebra $M$ with a choice of a faithful normal tracial state $\tau$. We denote by $L^2(M,\tau)$ the GNS Hilbert space corresponding to $\tau$ and identify $M$ with its representation on this space. We let $M^\circ=\{x^\circ\colon x\in M\}$ denote the opposite von Neumann algebra, represented on $L^2(M^\circ,\tau^\circ)$ which can be identified with the dual Hilbert space to $L^2(M,\tau)$. We let $\mmop$ denote the von Neumann algebra tensor product, which is equipped with the tensor product trace $\tau\otimes\tau^\circ$ and represented on $L^2(\mmop, \tau\otimes\tau^\circ)$. We will typically repress the `$\circ$' notation on elements of $\mmop$. We denote by $J_\tau$ and $J_{\tau\otimes \tau^\circ}$ the Tomita conjugation operators on $L^2(M,\tau)$ and $L^2(\mmop,\tau\otimes\tau^\circ)$, respectively, determined by $J_\tau(x)=x^*$ and $J_{\tau\otimes\tau^\circ} (a\otimes b)  = a^*\otimes b^*$ for $x,a,b\in M$. We also let $J_{\HS}$ denote the conjugate linear isometric involution on $L^2(\mmop,\ttop)$ determined by $J_\HS(a\otimes b) = b^*\otimes a^*$ for $a,b\in \H$, which we note is simply the usual adjoint on $\HS(L^2(M,\tau))\cong L^2(\mmop,\ttop)$.

Our fundamental objects of study will be unital $*$-subalgebras $B\subset A\subset M$ such that $A$ is finitely generated over $B$. In particular, given a finite subset $X=X^*\subset M$, $B\<X\>$ denotes the unital $*$-algebra generated by $X\cup B$. The von Neumann subalgebras of $M$ and $\mmop$ generated by $A$ and $A\otimes A^\circ$ are denoted $A''$ and $(A\otimes A^\circ)''$, respectively, and their closures in $L^2(M,\tau)$ and $L^2(\mmop,\tau\otimes\tau^\circ)$ are denoted $L^2(A,\tau)$ and $L^2(A\otimes A^\circ,\tau\otimes\tau^\circ)$, respectively. For a unital $*$-subalgebra $B\subset A$, we let $L^2_B(A\otimes A^\circ,\tau\otimes \tau^\circ)$ denote the subspace of $B$-central vectors.

\subsection{Free Calculus}

Let $(M,\tau)$ be a tracial von Neumann algebra with unital $*$-subalgebra $B\subset M$. Given a finite subset $X=X^*\subset M$, let $T_X=\{t_x\colon x\in X\}$ be a set of formal variables equipped with the involution $t_x^*:=t_{x^*}$. We denote by $B\<T_X\>$ the $*$-algebra formally spanned by elements of the form
    \[
        b_0t_{x_1}b_1\cdots b_{d-1}t_{x_d} b_d,
    \]
where $b_0,\ldots, b_d\in B$ and $x_1,\ldots, x_d\in X$. Let $\ev_X\colon B\<T_X\>\to B\<X\>$ be the $*$-homomorphism determined by
    \[
        \ev_X(b_0t_{x_1}b_1\cdots b_{d-1}t_{x_d} b_d) := b_0 x_1 b_1\cdots b_{d_1} x_{b_d} b_d.
    \]
Given $p\in B\<T_X\>$, we will write $p(X)$ for $\ev_X(p)$.

%

Let us write $A = B\ang{X}$, the subalgebra of $M$ generated by $B$ and $X$.
The \emph{free difference quotients} are the derivations $\partial_{x:B} : B\ang{T_X} \to \btxbtxop$ defined by linearity and the conditions
\begin{align*}
	\partial_{x:B}(t_y) &= \delta_{x=y}1\otimes1,\\
	\partial_{x:B}(pq) &= p\partial_{x:B}(q) + \partial_{x:B}(p)q.
\end{align*}
When $X$ admits no algebraic relations, these give rise to densely defined operators $L^2(A) \to L^2(\aaop)$; even when $X$ admits relations, they still give densely defined relations
    \[
        \set{(p(X), (\partial_{x:B}p)(X)) \colon p \in B\ang{T_X}},
    \]
which admit adjoints $L^2(\aaop) \to L^2(A)$.
Because the relations are densely defined, their adjoints are unbounded operators $\partial_{x:B}^* : L^2(\aaop) \to L^2(A)$, although they may not themselves be densely-defined.
Explicitly, $\xi \in \dom(\partial_{x:B}^*) \subset L^2(\aaop)$ if and only if there is $\eta \in L^2(A)$ so that for every $p \in B\ang{T_X}$,
\[\ang{\eta, p(X)}_\tau = \ang{\xi, \partial_{x:B}p}_{\ttop},\]
in which case $\eta = \partial_{x:B}^*\xi$.

We write $\partial_{X:B}$ for the \emph{free gradient},
\begin{align*}
	\partial_{X:B} : B\ang{T_X} &\longrightarrow (\btxbtxop)^X \\
	p &\longmapsto \paren{ \partial_{x:B}p }_{x\in X}.
\end{align*}
Likewise, for $d \in \N$ we define the \emph{free Jacobian}
\begin{align*}
	\J_{X:B} : B\ang{T_X}^d &\longrightarrow M_{d\times X}\paren{\btxbtxop} \\
	(p_i)_i &\longmapsto (\partial_{x:B}p_i)_{i,x}.
\end{align*}
As with $\partial_{x:B}$, these can be thought of as densely-defined relations from $L^2(A)$ to $L^2(\aaop)^X$ and from $L^2(A)^d$ to $M_{d\times\abs X}\paren{ L^2(\aaop) }$, respectively, and they have corresponding unbounded (though not necessarily densely-defined) adjoints $\partial_{X:B}^*$ and $\J_{X:B}^*$.
When $B=\C$, we will elide it from the subscript.

We now recall some concepts from \cite{CN21}.
Let $\1$ denote the identity matrix in $M_{\abs X \times \abs X}(L^2(\aaop))$.
The \emph{free Stein irregularity} of $X$, denoted $\Sigma^*(X:B)$, is the distance from $\1$ to $\overline{\dom \J_{X:B}^*}$ measured in the Hilbert-Schmidt norm on $M_{\abs X \times \abs X}(L^2(\aaop))$.
Meanwhile, the \emph{free Stein dimension} is the quantity $\sigma(X:B) = \dim_{\aaop}\overline{\dom\paren{ \partial_{X:B}^* }}$.
These two are related by the equation $\sigma(X:B) = \abs{X} - \Sigma^*(X:B)^2$; see \cite[Theorem~2.11, Definition~2.12]{CN21}.

\subsection{Derivation Spaces and Free Stein Dimension}\label{subsec:derivation_spaces}

Let $A\subset M$ be a unital $*$-subalgebra. We denote by $\Der(A,\tau)$ the vector space of derivations $\delta\colon A \to L^2(A\otimes A^\circ,\tau\otimes \tau^\circ)$. We will also consider the following subspaces:
    \begin{itemize}
        \item $\InnDer(A,\tau)$: the inner derivations $\delta=[\cdot,\xi]$ for some $\xi\in L^2(A\otimes A^\circ,\tau\otimes\tau^\circ)$;
        \item $\Der_{1\otimes 1}(A,\tau)$: the derivations $\delta$ with $1\otimes 1\in \dom(\delta^*)$ when $\delta$ is viewed as a densely defined operator $L^2(A,\tau)\to L^2(A\otimes A^\circ,\tau\otimes \tau^\circ)$.
    \end{itemize}
If $B\subset A$ is a unital $*$-subalgebra, then we write $\Der(B\subset A,\tau)$ for the subspace of derivations vanishing on $B$ and $\InnDer(B\subset A,\tau)$ and $\Der_{1\otimes 1}(B\subset A,\tau)$ for the corresponding subspaces.  Note that $\Der(A,\tau) = \Der(\C\subset A,\tau)$.

The condition $1\otimes 1\in \dom(\delta^*)$ for $\delta\in \Der_{1\otimes 1}(A,\tau)$ implies $A\otimes A^\circ\subset \dom(\delta^*)$ by the same proof as \cite[Proposition 4.1]{VoiV}. Consequently, $\delta\colon L^2(A,\tau)\to L^2(\aaop,\ttop)$ is closable as a densely defined operator. This fact will be used implicitly in the rest of the paper. The closure of $\delta$ will be denoted $\overline{\delta}$. We also note that $\InnDer(A,\tau)\subset \Der_{1\otimes 1}(A,\tau)$, where
    \[
        ([\cdot,\xi])^*(1\otimes 1) = (1\otimes \tau - \tau\otimes 1)(J_{\ttop}\xi)
    \]
for any $\xi\in L^2(\aaop,\ttop)$.

Recall that we say $\delta\in \Der(A,\tau)$ is \emph{real} if $J_\HS\delta(x^*)=\delta(x)$ for all $x\in A$. Equivalently,
    \[
        \< x\cdot \delta(y), \delta(z)\>_{\ttop} = \< \delta(z^*), \delta(y^*)\cdot x^*\>_{\ttop}
    \]
for all $x,y,z\in A$. For any $\delta\in \Der(A,\tau)$, we can define $\delta^\dagger\in \Der(A,\tau)$ by $\delta^\dagger(x):=J_\HS\delta(x^*)$. Then
    \[
        \text{Re}(\delta):= \frac{1}{2}(\delta+\delta^\dagger) \qquad \text{ and }\qquad \text{Im}(\delta):= \frac{1}{2i}(\delta - \delta^\dagger)
    \]
are both real derivations. In particular, if $\delta\in \Der_{1\otimes 1}(B\subset A,\tau)$ then $\delta^\dagger, \text{Re}(\delta), \text{Im}(\delta)\in \Der_{1\otimes 1}(B\subset A,\tau)$ with
    \begin{align*}
        (\delta^\dagger)^*(1\otimes 1) &= J_\tau \delta^*(1\otimes 1)\\
        (\text{Re}(\delta))^*(1\otimes 1) &= \frac{1}{2}(1+J_\tau)\delta^*(1\otimes 1)\\
        (\text{Im}(\delta))^*(1\otimes 1) &=\frac{i}{2}(1- J_\tau)\delta^*(1\otimes 1).
    \end{align*}

The right action of $(A\otimes A^\circ)''$ on $L^2(A\otimes A^\circ,\tau\otimes \tau^\circ)$ induces a right action on $\Der(B\subset A,\tau)$:
    \[
        (\delta\cdot m)(x):=\delta(x)m
    \]
for $m\in (\aaop)''$ and $x\in A$. In particular, $\InnDer(B\subset A,\tau)$ is invariant under the right action of $(\aaop)''$ while $\Der_{1\otimes 1}(B\subset A,\tau)$ is only invariant under the right action of $\aaop$. A priori, this is just an algebraic module structure, but we will see below that it can be made into a Hilbert module.

\begin{lem}\label{lem:derivations_as_Hilbert_spaces}
Let $(M,\tau)$ be a tracial von Neumann algebra with unital $*$-subalgebras $B\subset A\subset M$ such that $A$ is finitely generated over $B$. Let $X=X^*\subset A$ be any finite subset such that $A=B\<X\>$. Define a linear map
	\begin{align*}
		\phi_X\colon \Der(B\subset A,\tau)&\to L^2(A\otimes A^\circ,\tau\otimes \tau^\circ)^{X} \\
			\delta &\mapsto (\delta(x))_{x\in X}.
	\end{align*}
Then $\phi_X$ is injective, right $(A\otimes A^\circ)''$-linear, and has closed range. Moreover, if $\ev_X\colon B\<T_X\>\to B\<X\>$ is injective then $\phi_X$ is surjective.
\end{lem}
\begin{proof} Any $\delta\in \Der(B\subset A,\tau)$ is uniquely determined by its values on $X$. Hence $\phi_X$ is injective. Also $\phi_X(\delta\cdot m)=\phi_X(\delta)\cdot m$ for $m\in (A\otimes A^\circ)''$ acting diagonally. Next, suppose
	\[
		(\xi_x)_{x\in X} = \lim_{n\to\infty} (\delta_n(x))_{x\in X}
	\]
for some sequence $(\delta_n)_{n\in \N} \subset \Der(B\subset A,\tau)$. This implies for all $p\in B\<T_X\>$ that
	\[
		\delta_n(p(X)) = \sum_{x\in X} \ev_X(\partial_{x\colon B} p) \# \delta_n(x) \to \sum_{x\in X} \ev_X(\partial_{x\colon B} p)\# \xi_x.
	\]	
Thus we can define a map $\delta\colon A\to L^2(A\otimes A^\circ, \tau\otimes \tau^\circ)$ by
	\[
		\delta(p(X))) =\lim_{n\to\infty} \delta_n(p(X)).
	\]
It follows that $\delta$ satisfies the Leibniz rule and vanishes on $B$. That is, $\delta\in \Der(B\subset A,\tau)$. Also, $\phi_X(\delta) = (\xi_x)_{x\in X}$, so the range of $\phi_X$ is closed.

If $\ev_X$ is injective, then $\partial_{x\colon B} \in \Der(B\subset A,\tau)$ for all $x\in X$. Then given any $(\xi_x)_{x\in X}\subset L^2(A\otimes A^\circ,\tau\otimes \tau^\circ)^{X}$, we can define
	\[
		\delta:= \sum_{x\in X} \partial_{x\colon B} \# \xi_x
	\]
and obtain $\phi_X(\delta) = (\xi_x)_{x\in X}$.
\end{proof}

The above lemma implies that $\Der(B\subset A,\tau)$ is a Hilbert space under the inner product
	\[
		\<\delta_1,\delta_2\>_X:= \sum_{x\in X} \<\delta_1(x), \delta_2(x)\>_{\tau\otimes\tau^\circ}.
	\]
Note that convergence with respect to the induced norm is equivalent to pointwise convergence on $X$, and consequently equivalent to pointwise convergence on $A$. Moreover, $\Der(B\subset A,\tau)$ is a closed subspace of $\Der(A,\tau)$ with respect to this norm. 

Let $Y=Y^*\subset A$ be a finite subset satisfying $B\<Y\>=A$. For each $y\in Y$, there exists $p_y\in B\<T_X\>$ such that $y=p_y(X)$. Consequently, for $\delta\in \Der(B\subset A,\tau)$ we have
    \[
       \|\delta\|_Y^2 \leq \sum_{\substack{x\in X\\ y\in Y}} \| (\partial_{x\colon B} p_y)(X) \|^2 \|\delta\|_X^2.
    \]
Similarly, one can bound $\|\delta\|_X$ by a scalar multiple of $\|\delta\|_Y$; that is, the norms are equivalent. Consequently, for a subset $S\subset \Der(B\subset A,\tau)$, we let $\overline{S}$ denote its closure with respect to any generating set $X=X^*$ over $B$.

For any closed  $(\aaop)''$-invariant subspace $\H\leq \Der(B\subset A,\tau)$, Lemma~\ref{lem:derivations_as_Hilbert_spaces} also implies $\phi_X(\H)$ is a right Hilbert $(A\otimes A^\circ)''$-submodule of $L^2(A\otimes A^\circ,\tau\otimes \tau^\circ)^X$, and so one can compute its von Neumann dimemsion. It turns out this is independent of the generating set $X$:

\begin{lem}
Let $X,Y\subset A$ be self-adjoint finite subsets satisfying $B\<X\>=B\<Y\>=A$. Then for any closed $(\aaop)''$-invariant subspace $\H\leq \Der(B\subset A,\tau)$, the von Neumann dimensions of $\phi_X(\H)$ and $\phi_Y(\H)$ as right Hilbert $(A\otimes A^\circ)''$-modules agree:
    \[
        \dim \phi_X(\H)_{(A\otimes A^\circ)''} = \dim \phi_Y(\H)_{(A\otimes A^\circ)''}.
    \]
\end{lem}
\begin{proof}
The map 
    \[
        \phi_Y\circ \phi_X^{-1}\colon \phi_X(\H) \to \phi_Y(\H)
    \]
is a bijection by Lemma~\ref{lem:derivations_as_Hilbert_spaces}. Moreover,
    \[
        \| \phi_Y\circ \phi_X^{-1}(\phi_X(\delta)) \|_Y = \|\delta\|_Y\leq c \|\delta\|_X,
    \]
for some $c>0$ since the norms $\|\cdot\|_X$ and $\|\cdot \|_Y$ are equivalent. Thus $\phi_Y\circ \phi_X^{-1}$ is continuous, and by symmetry so is its inverse $\phi_X\circ \phi_Y^{-1}$. Therefore it is a right Hilbert $(A\otimes A^\circ)''$-module isomorphism and the dimensions agree.
\end{proof}

In light of the previous lemma, we make the following definition:

\begin{defi}
Let $(M,\tau)$ be a tracial von Neumann algebra with unital $*$-subalgebras $B\subset A\subset M$ such that $A$ is finitely generated over $B$. For a closed $(\aaop)''$-invariant subspace $\H\leq \Der(B\subset A,\tau)$,
    \[
        \dim \H_{(A\otimes A^\circ)''}:= \dim \phi_X(\H)_{(A\otimes A^\circ)''},
    \]
where $X=X^*\subset A$ is any finite subset satisfying $B\<X\>=A$.
\end{defi}

An important first example comes from the space of inner derivations, or rather its closure which is a right Hilbert $(\aaop)''$-module. Recall that $L^2_B(\aaop,\ttop)$ and $L^2_A(\aaop,\ttop)$ denote the $B$-central and $A$-central vectors in $L^2(A\otimes A^\circ,\tau\otimes \tau^\circ)$, respectively, which are right $(A\otimes A^\circ)''$-modules.

\begin{lem}\label{lem:derivations_from_non-A-central_vectors}
Let $(M,\tau)$ be a tracial von Neumann algebra with unital $*$-subalgebras $B\subset A\subset M$ such that $A$ is finitely generated over $B$. For any 
    \[
        \H_{(A\otimes A^\circ)''} \leq L^2_B(A\otimes A^\circ,\tau\otimes \tau^\circ)\ominus L^2_A(A\otimes A^\circ,\tau\otimes \tau^\circ)
    \]
one has
    \[
        \dim \overline{\{[\cdot, \xi]\colon \xi\in \H\}}_{(A\otimes A^\circ)''}  = \dim \H_{(A\otimes A^\circ)''}.
    \]
\end{lem}
\begin{proof}
Let $p\in (A\otimes A^\circ)''$ be the projection such that $\H=pL^2(A\otimes A^\circ,\tau\otimes\tau^\circ)$, and consider the projective right $(A\otimes A^\circ)''$-module
    \[
        P:=\{[\cdot, pm]\colon m\in (A\otimes A^\circ)''\}\subset \Der(B\subset A,\tau).
    \]
The map $p(A\otimes A^\circ)''\ni pm\mapsto [\cdot,pm]\in P$ is injective since the derivation $[\cdot,pm]$ being zero implies $pm$ is $A$-central and hence zero. Thus this map is a right $(A\otimes A^\circ)''$-module isomorphism and therefore the dimensions of its domain and range agree (see \cite[Equation 6.3]{Luc02}):
    \[
        \dim P_{(A\otimes A^\circ)''} = \dim p(A\otimes A^\circ)''_{(A\otimes A^\circ)''}= \tau\otimes \tau^\circ(p) = \dim \H_{(A\otimes A^\circ)''}.
    \]
On the other hand, if $X\subset A$ is a generating set then $P$ has an inner product given by $\<\cdot,\cdot\>_X$. So \cite[Theorem 6.24]{Luc02} and the density of $p(A\otimes A)''$ in $\H$ imply
    \[
        \dim P_{(A\otimes A^\circ)''} = \dim \overline{P}_{(A\otimes A^\circ)''} = \dim \overline{\{[\cdot, \xi] \colon \xi\in \H\}}_{(A\otimes A^\circ)''}.\qedhere
    \]
\end{proof}

Also of interest for this paper is the right Hilbert $(\aaop)''$-module of closed subspace $\overline{\Der_{1\otimes 1}(B\subset A,\tau)}$. Recall that $\Der_{1\otimes 1}(B\subset A,\tau)$ is $\aaop$-invariant, and consequently its closure is $(\aaop)''$-invariant by the Kaplansky density theorem. The dimension of this $(\aaop)''$-module is the focus of this paper:

\begin{defi}\label{def:fsd}
Let $(M,\tau)$ be a tracial von Neumann algebra, and let $B\subset A\subset M$ be unital $*$-subalgebras with $A$ finitely generated over $B$. The \textbf{free Stein dimension} of $A$ over $B$ with respect to $\tau$ is the quantity
    \[
        \sigma(B\subset A,\tau):= \dim \overline{\Der_{1\otimes 1}(B\subset A,\tau)}_{(A\otimes A^\circ)''}.
    \]
For $B=\C$, we simply write $\sigma(A,\tau)$.
\end{defi}

This quantity is equivalent to the quantity defined in \cite{CN21}. Let $X=X^*\subset A$ be a finite subset with $B\<X\>=A$. Then $\sigma(X\colon B)$ is the \emph{free Stein dimension} of $X$ over $B$ defined in \cite[Definition 2.12]{CN21}, and by \cite[Lemma 3.1]{CN21}
    \begin{align*}
        \sigma(X\colon B) &= \dim {_{(B\<X\>\otimes B\<X\>^\circ)''}}\overline{\{(J_{\tau\otimes \tau^\circ}\delta(x))_{x\in X}\colon \delta\in \Der_{1\otimes 1}(B\subset B\<X\>,\tau)\}}\\
            &= \dim \overline{\Der_{1\otimes 1}(B\subset B\<X\>,\tau)}_{(B\<X\>\otimes B\<X\>^\circ)''} = \sigma(B\subset A,\tau).
    \end{align*}
Thus $\sigma(B\subset A,\tau)$ is not a new quantity. Instead, we are merely attempting to adopt a more natural perspective in terms of derivations, which at the very least simplifies the exposition of the paper.

\begin{rem}
In \cite{CS05}, Connes and Shlyakhtenko defined $L^2$-Betti numbers $\beta_k^{(2)}(A,\tau)$ associated to tracial $*$-algebras. For $k=0,1$ these can be expressed in the above notation as:
    \begin{align*}
        \beta_0^{(2)}(A,\tau) &= 1 - \dim \InnDer(A,\tau)_{(A\otimes A^\circ)''},\\
        \beta_1^{(2)}(A,\tau) &= \dim \Der(A,\tau)_{(A\otimes A^\circ)''} - \dim \InnDer(A,\tau)_{(A\otimes A^\circ)''}
    \end{align*}
(compare the end of the proof of \cite[Lemma 3.1]{Shl21}). Consequently, one always has
    \[
        \sigma(A,\tau) \leq \dim \Der(A,\tau)_{(A\otimes A^\circ)''} = \beta_1^{(2)}(A,\tau) - \beta_0^{(2)}(A,\tau) +1.
    \]
If $\sigma(A,\tau)=|X|$ some finite generating set $X=X^*\subset A$, then Lemma~\ref{lem:derivations_as_Hilbert_spaces} implies the above is an equality. This happens precisely when $X$ has zero free Stein irregularity, for example whenever $\Phi^*(X)<\infty$. If $A''$ is diffuse so that $L^2_A(A\otimes A^\circ,\tau\otimes \tau^\circ)=\{0\}$, then Lemma~\ref{lem:derivations_from_non-A-central_vectors} implies $\beta_0^{(2)}(A,\tau)=0$. So in this case the above inequality reduces to $\sigma(A,\tau)\leq \beta_1^{(2)}(A,\tau) +1$.
\end{rem}


\section{Direct Sum and Amplification Formulae}\label{sec:direct_sum_and_amplification_formulae}

In this section we show that the free Stein dimension satisfies the expected direct sum and amplification formulas (see Theorems~\ref{thm:direct_sum_formula} and \ref{thm:amplification_formula}). The key unifying feature in these proofs is that the algebras contain finite dimensional subalgebras.

\begin{thm}\label{thm:excising_a_finite_dimensional_subalgebra}
Let $(M,\tau)$ be a tracial von Neumann algebra with finitely generated unital $*$-subalgebra $A\subset M$. For a finite dimensional unital $*$-subalgebra $B\subset A$, one has
    \[
        \sigma(A,\tau) = \sigma(B\subset A, \tau) + \sigma(B,\tau).
    \]
\end{thm}
\begin{proof}
Our strategy will be to show that for a carefully chosen generating set (and hence inner product), the orthogonal complement of $\Der_{1\otimes 1}(B\subset A,\tau)$ in $\Der_{1\otimes 1}(A,\tau)$ is isomorphic to $L^2_B(\aaop, \ttop)^\perp$.

Fix a finite set $X=X^*\subset A$ satisfying $A=\C\<X\>$ and fix a finite dimensional $B\subset A$. Then
    \[
        (B,\tau) = \left( \sum_{i=1}^d M_{n_i}(\C), \sum_{i=1}^d \alpha_i \tr_{n_i} \right)
    \]
for some $n_1,\ldots, n_d\in \N$ and $\alpha_1,\ldots, \alpha_d>0$ satisfying $\sum_{i=1}^d \alpha_i =1$. Let $E:=\{e_{j,k}^{(i)}\colon 1\leq i \leq d,\ 1\leq j,k\leq n_i\}$ be a family of multi-matrix units for $B$. Then
    \[
        p:= \sum_{i=1}^d \frac{1}{n_i} \sum_{j,k=1}^{n_i} e_{j,k}^{(i)} \otimes e_{k,j}^{(i)}
    \]
is the projection from $L^2(\aaop,\ttop)$ onto the $B$-central vectors $L^2_B(\aaop,\ttop)$. 

Consider the self-adjoint finite subset
    \[
        Y:=\{ e x e'\colon x\in X,\ e, e'\in E\} \subset A.
    \]
Since $1\in \text{span}(E)$, it follows that $\C\<Y\>=A$. Observe that for $\delta\in \Der(B\subset A,\tau)$,
    \[
        \sum_{y\in Y} y^*\delta(y) = \sum_{\substack{x\in X\\ e,e'\in E}} (e')^* x^* e^* \delta( exe') = \sum_{i=1}^d \sum_{j,k=1}^{n_i} (e_{j,k}^{(i)}\otimes e_{k,j}^{(i)})\sum_{\substack{x\in X\\e\in E}} x^*e^*e\delta(x),
    \]
which is $B$-central. Similarly for $\sum_{y\in Y} \delta(y)y^*$. Thus for any $\xi\in L^2(\aaop,\ttop)$ one has
    \[
        \< [\cdot, \xi], \delta\>_Y = \sum_{y\in Y} \< [y,\xi], \delta(y)\>_{\ttop} = \sum_{y\in Y} \< \xi, [y^*,\delta(y)]\>_\ttop = \sum_{y\in Y} \< p\xi, [y^*,\delta(y)]\>_\ttop = \<[\cdot, p\xi], \delta\>_Y.
    \]
Therefore $[\cdot, p\xi]$ is the projection of $[\cdot,\xi]$ onto $\Der(B\subset A,\tau)$.

Now, let $\delta\in \Der_{1\otimes 1}(A,\tau)$. Since $B$ is finite dimensional, there exists $\xi\in L^2(\aaop,\ttop)$ with $\delta|_B = [\cdot, \xi]|_B = [\cdot, (1-p)\xi]|_B$. The above arguments imply
    \[
        \delta = (\delta - [\cdot, (1-p)\xi]) + [\cdot, (1-p)\xi]
    \]
is an orthogonal decomposition with respect to $\<\cdot,\cdot\>_Y$, and we note that the first term belongs to $\Der_{1\otimes 1}(B\subset A,\tau)$. It follows that 
    \[
     \Der_{1\otimes 1}(A,\tau) = \Der_{1\otimes 1}(B\subset A,\tau) \oplus \{[\cdot,\xi]\colon \xi\in L^2_B(A\otimes A^\circ,\tau\otimes\tau^\circ)^\perp\}.
    \]
Since $L_B^2(\aaop,\ttop)^\perp\leq L_A^2(\aaop,\ttop)^\perp$, taking the von Neumann dimensions of the closures in the above orthogonal decomposition and using Lemma~\ref{lem:derivations_from_non-A-central_vectors} yields
    \[
        \sigma(A,\tau) = \sigma(B\subset A,\tau) + \dim L^2_B(A\otimes A^\circ,\tau\otimes\tau^\circ)^\perp_{(\aaop)''}.
    \]
Finally, we have
    \begin{align*}
        \dim L^2_B(A\otimes A^\circ,\tau\otimes\tau^\circ)^\perp_{(\aaop)''} = \ttop(1-p) = 1 - \sum_{i=1}^d \frac{1}{n_i} \sum_{j=1}^{n_i} \tau(e_{j,j}^{(i)})^2 = 1 - \sum_{i=1}^d \frac{\alpha_i^2}{n_i^2} = \sigma(B,\tau),
    \end{align*}
where the last equality follows from \cite[Corollary 5.2]{CN21}.
\end{proof}

\begin{rem}
Since the decomposition
    \[
        \delta = (\delta - [\cdot,(1-p)\xi]) + [\cdot, (1-p)\xi], \qquad \delta\in \Der(A,\tau)
    \]
from the above proof is unique and independent of the chosen generating set, one might hope give a more direct proof using the algebraic dimension from \cite{Luc02}. Indeed, the uniqueness of this decomposition shows $\Der(B\subset A,\tau)\cap \overline{\Der_{1\otimes 1}(A,\tau)}$ is algebraically complemented in $\overline{\Der_{1\otimes 1}(A,\tau)}$. However, the trouble is that this intersection need not equal $\overline{\Der_{1\otimes 1}(B\subset A,\tau)}$ for general $B$. The fact that it does for finite dimensional $B$ follows from our above proof.
\end{rem}

\begin{ex}\label{ex:free_graph_algebras}
Fix a finite, connected graph $\Gamma= (V,E)$  with vertex weighting $\mu\colon V\to [0,1]$ satisfying $\sum_{v\in V} \mu(v)=1$, and let $\vec{\Gamma}=(V,\vec{E})$ be the associated directed graph (cf. \cite{HN18}). Recall that the free graph von Neumann algebra $(\mc{M}(\Gamma,\mu),\tau)$ is generated by operators $X:=\{x_\epsilon\colon \epsilon\in \vec{E}\}$ and an orthogonal family of projections $Y:=\{p_v\colon v\in V\}$, which satisfy the following graph relations:
	\begin{itemize}
		\item $\tau(p_v) = \mu(v)$ for all $v\in V$;

		\item $x_{\epsilon}^* = x_{\epsilon^\text{op}}$ for all $\epsilon\in \vec{E}$;

		\item $p_v x_{\epsilon} p_w = \delta_{v=s(\epsilon)} \delta_{w=t(\epsilon)} x_{\epsilon}$ for all $v,w\in V$ and $\epsilon \in \vec{E}$.
	\end{itemize}
Moreover, if $\mu(v)\leq \sum_{w\sim v} n_{v.w} \mu(w)$ for all $v\in V$ where $n_{v,w}$ is the number of edges connecting $v$ to $w$, then there is a trace-preserving isomorphism between $M$ and the interpolated free group factor $L(\mathbb{F}_t)$ with parameter
	\[
		t:=1-\sum_{v\in V} \mu(v)^2 + \sum_{v\in V} \mu(v) \sum_{w\sim v} n_{v,w}\mu(w).
	\]
It was shown in \cite[Example B.2]{CN21} that
    \[
        \sigma(B\subset A, \tau) + \sigma(B,\tau)=t
    \]
for $A:=\C\<X\cup Y\>$ and $B:=\C\<Y\>$. Using Theorem~\ref{thm:excising_a_finite_dimensional_subalgebra} we obtain $\sigma(A,\tau)=t$.$\hfill\blacksquare$
\end{ex}

\begin{thm}\label{thm:direct_sum_formula}
Let $\{(M_i,\tau_i)\}_{i=1}^d$ be a family of tracial von Neumann algebras. For any finitely generated unital $*$-subalgebras $A_i\subset M_i$, $i=1,\ldots, d$, and any scalars $\alpha_1,\ldots, \alpha_d> 0$ satisfying $\sum_{i=1}^d \alpha_i=1$, one has
    \[
        \sigma\left( \bigoplus_{i=1}^d A_i, \sum_{i=1}^d \alpha_i \tau_i \right) = \sum_{i=1}^d \alpha_i^2 \sigma(A_i, \tau_i) + \alpha_i(1-\alpha_i).
    \]
\end{thm}
\begin{proof}
Denote $M:=\bigoplus_{i=1}^d M_i$ and $\tau:=\sum_{i=1}^d \tau_i$. Let $X_i=X_i^*\subset A_i$ be a generating set and let $1_i\in A_i$ be the unit for each $i=1,\ldots, d$. Then $X=X_1\cup\cdots \cup X_d$ and $Y=\{1_1,\ldots, 1_d\}$ generate $A:=\bigoplus_{i=1}^d A_i$, and $B:=\C\<Y\>\cong \C^d$, respectively. Theorem~\ref{thm:excising_a_finite_dimensional_subalgebra} implies
    \[
        \sigma(A,\tau)=\sigma(B\subset A,\tau) + \sigma(B,\tau),
    \]
so we will compute each term separately. For $\delta\in \Der_{1\otimes 1}(B\subset A,\tau)$, observe that
    \[
        \delta(x) = \sum_{i=1}^d \delta(1_i x 1_i) = \sum_{i=1}^d 1_i \delta(x) 1_i = \sum_{i=1}^d \delta(x)(1_i\otimes 1_i).
    \]
This gives the identification
    \[
        \Der_{1\otimes 1}(B\subset A,\tau) =\bigoplus_{i=1}^d \Der_{1\otimes 1}(A_i,\tau_i),
    \]
where the direct sum is with respect to the $\<\cdot, \cdot\>_{X\cup Y}$ inner product. Thus taking closures yields
    \[
        \sigma(B\subset A,\tau) = \dim \bigoplus_{i=1}^d \overline{\Der_{1\otimes 1}(A_i,\tau_i)}_{M\otimes M^\circ} = \sum_{i=1}^d \alpha_i^2 \dim \overline{\Der_{1\otimes 1}(A_i,\tau_i)}_{M_i\otimes M_i^\circ} = \sum_{i=1}^d \alpha_i^2 \sigma(A_i,\tau_i).
    \]
Using \cite[Corollary 5.2]{CN21} we have
    \[
        \sigma(B,\tau) = 1 - \sum_{i=1}^d \alpha_i^2 = \sum_{i=1}^d \alpha_i(1-\alpha_i)
    \]
and the claimed formula for $\sigma(A,\tau)$ follows.
\end{proof}

\begin{thm}\label{thm:amplification_formula}
Let $\{(M_i,\tau_i)\}_{i=1}^d$ be a family of tracial von Neumann algebras. For any finitely generated unital $*$-subalgebras $A_i\subset M_i$, $i=1,\ldots, d$, any scalars $\alpha_1,\ldots, \alpha_d> 0$ satisfying $\sum_{i=1}^d \alpha_i=1$, and any $n_1,\ldots, n_d\in\N$, one has
    \[
        \sigma\left( \bigoplus_{i=1}^d A_i\otimes M_{n_i}(\C), \sum_{i=1}^d \alpha_i \tau_i \otimes \tr_{n_i} \right) = 1+\sum_{i=1}^d \frac{\alpha_i^2}{n_i^2} (\sigma(A_i,\tau_i)-1) .
    \]
\end{thm}
\begin{proof}
Observe that the $d=1$ case says $\sigma(A\otimes M_n(\C),\tau\otimes \tr_n) = 1+ \frac{1}{n^2}(\sigma(A) - 1)$. If this holds, then using Theorem~\ref{thm:direct_sum_formula} we have
    \begin{align*}
        \sigma\left( \bigoplus_{i=1}^d A_i\otimes M_{n_i}(\C), \sum_{i=1}^d \alpha_i \tau_i \otimes \tr_{n_i} \right) &= \sum_{i=1}^d \alpha_i^2 \sigma( A_i\otimes M_{n_i}(\C), \tau_i\otimes\tr_{n_i}) + \alpha_i(1-\alpha_i)\\
        &= \sum_{i=1}^d \alpha_i^2 + \frac{\alpha_i^2}{n_i^2}(\sigma(A_i,\tau_i) -1) + \alpha_i - \alpha_i^2\\
        &= 1+ \sum_{i=1}^d \frac{\alpha_i^2}{n_i^2} (\sigma(A_i,\tau_i) - 1).
    \end{align*}
Thus it suffices to prove the $d=1$ case.

Let $X=X^*\subset A$ be a generating set, let $E=\{e_{i,j}\}_{i,j=1}^n$ be a family of matrix units for $M_n(\C)$, and denote $\tau_n:=\tau\otimes \tr_n$. Theorem~\ref{thm:excising_a_finite_dimensional_subalgebra} and \cite[Corollary 5.2]{CN21} imply
    \begin{align}\label{eqn:amplification_decomp}
        \sigma(A\otimes M_n(\C),\tau_n)  &= \sigma(M_n(\C) \subset M_n(A),\tau_n) + \sigma(M_n(\C),\tr_n)\nonumber\\
            &=\sigma(M_n(\C) \subset M_n(A),\tau_n) + 1 - \frac{1}{n^2}.
    \end{align}
Now, observe that since $(M_n(A)\otimes M_n(A)^\circ)'' \cong M_{n^2}((A\otimes A^\circ)'')$ and $M_n(\C)\otimes M_n(C)^\circ \subset (A\otimes A^\circ)'$ we have
    \begin{align*}
        \sigma(M_n(\C)\subset M_n(A),\tau_n) &=\dim \overline{\Der_{1\otimes 1}(M_n(\C)\subset M_n(A),\tau_n)}_{W^*(M_n(A)\otimes M_n(A)^\circ)}\\ 
            &= \frac{1}{n^4} \dim \overline{\Der_{1\otimes 1}(M_n(\C)\subset M_n(A),\tau_n)}_{(A\otimes A^\circ)''}\\
            &= \frac{1}{n^4} \dim \bigoplus_{i,j=1}^n  \overline{\Der_{1\otimes 1}(M_n(\C)\subset M_n(A),\tau_n)}(e_{i,i}\otimes e_{j,j})_{(A\otimes A^\circ)''}\\
            &= \frac{1}{n^4} \sum_{i,j=1}^n \dim  \overline{\Der_{1\otimes 1}(M_n(\C)\subset M_n(A),\tau_n)}(e_{i,i}\otimes e_{j,j})_{(A\otimes A^\circ)''}.
    \end{align*}
We will show $\overline{\Der_{1\otimes 1}(M_n(\C)\subset M_n(A),\tau_n)}(e_{i,i}\otimes e_{j,j})\cong \overline{\Der_{1\otimes 1}(A,\tau)}$ for each $i,j=1,\ldots, n$, so that
    \begin{align}\label{eqn:amplification_over_matrices}
        \sigma(M_n(\C)\subset M_n(A),\tau_n) = \frac{1}{n^4} \sum_{i,j=1}^n \dim \overline{\Der_{1\otimes 1}(A,\tau)}_{(A\otimes A^\circ)''} = \frac{1}{n^2} \sigma(A,\tau).
    \end{align}
In light of (\ref{eqn:amplification_decomp}), this will complete the proof.

Given $\delta\in \Der_{1\otimes 1}(A,\tau)$ and $1\leq a,b\leq n$, define $\delta\otimes e_{a,b}\in \Der_{1\otimes 1}(M_n(\C)\subset M_n(A),\tau_n)(e_{a,a}\otimes e_{b,b})$ by
	\[
		[\delta\otimes e_{a,b}](x\otimes e_{i,j}):= \delta(x)\otimes e_{i,a}\otimes e_{b,j}.
	\]
(Here we are identifying $L^2(M_n(A)\otimes M_n(A)^\circ,\tau_n\otimes\tau_n^\circ)\cong L^2(A\otimes A^\circ,\tau\otimes\tau^\circ)\otimes M_n(\C)\otimes M_n(\C)^\circ$.) This vanishes on $M_n(\C)$ since $\delta(1)=0$, and is easily checked to satisfy the Leibniz rule as well as
	\begin{align*}
		[\delta\otimes e_{a,b}]^*(1\otimes 1\otimes 1\otimes 1) = \frac1n \delta^*(1\otimes 1)\otimes e_{a,b}.
	\end{align*}
For $\Delta\in \Der_{1\otimes 1}(M_n(\C)\subset M_n(A),\tau_n)(e_{a,a}\otimes e_{b,b})$, define
    \[
        \delta(x) := (1\otimes \Tr)\otimes (1\otimes \Tr)^\circ\left[ \Delta(x\otimes e_{a,b})\right].
    \]
Then $\delta\in \Der_{1\otimes 1}(A,\tau)$ with
    \[
        \delta^*(1\otimes 1) = n (1\otimes \Tr)[(1\otimes e_{b,a})\Delta^*(1\otimes 1\otimes 1\otimes 1)].
    \]
Observe that for any $1\leq i,j\leq n$ we have
    \begin{align*}
        \delta(x)&= (1\otimes \Tr)\otimes (1\otimes \Tr)^\circ\left[ (e_{a,i}\otimes e_{j,b}) \Delta(x\otimes e_{i,j})\right]\\
            &=(1\otimes \Tr)\otimes (1\otimes \Tr)^\circ\left[ \Delta(x\otimes e_{i,j})( e_{a,i}\otimes e_{j,b})\right].
    \end{align*}
Since $\Delta(x\otimes e_{i,j}) = (e_{i,i}\otimes e_{j,j})\Delta(x\otimes e_{i,j}) (e_{a,a}\otimes e_{b,b})$, it follows that
    \begin{align*}
        \Delta(x\otimes e_{i,j}) &= \sum_{c,d,k,\ell=1}^n (1\otimes \Tr)\otimes (1\otimes \Tr)^\circ[ \Delta(x\otimes e_{i,j})(e_{c,k}\otimes e_{\ell,d})] (e_{k,c}\otimes e_{d,\ell})\\
            &=(1\otimes \Tr)\otimes (1\otimes \Tr)^\circ[ \Delta(x\otimes e_{i,j})(e_{a,i}\otimes e_{j,b})] (e_{i,a}\otimes e_{b,j})\\
            &= \delta(x)\otimes e_{i,a}\otimes e_{b,j}= [\delta\otimes e_{a,b}](x\otimes e_{i,j}).
    \end{align*}
Thus $\Delta=\delta\otimes e_{a,b}$, and so the right $(A\otimes A^\circ)''$-module map $\Der_{1\otimes 1}(A,\tau)\ni\delta\mapsto \delta\otimes  e_{a,b}\in \Der_{1\otimes 1}(M_n(\C)\subset M_n(A),\tau_n)$ is surjective. One also readily computes that
    \[
        \|\delta\otimes e_{a,b}\|_{(X\otimes 1, 1\otimes E)}  = \frac{1}{\sqrt{n}} \|\delta\|_X.
    \]
Hence this map extends to a right $(A\otimes A^\circ)''$-module isomorphism $\overline{\Der_{1\otimes 1}(M_n(\C)\subset M_n(A),\tau_n)}(e_{a,a}\otimes e_{b,b})\cong \overline{\Der_{1\otimes 1}(A,\tau)}$.
\end{proof}

\begin{cor}\label{cor:tensor_against_finite_dimensional_formula}
Let $(M,\tau)$ and $(B,\phi)$ be tracial von Neumann algebras with $B$ finite dimensional. For any finitely generated unital $*$-subalgebra $A\subset M$, one has
    \[
        \sigma(A\otimes B,\tau\otimes\phi) = \sigma(A,\tau) + \sigma(B,\phi) - \sigma(A,\tau)\sigma(B,\phi).
    \]
\end{cor}
\begin{proof}
$B$ is necessarily a multimatrix algebra:
    \[
        (B,\phi)=\left( \bigoplus_{i=1}^d M_{n_i}(\C), \sum_{i=1}^d \alpha_i \tr_{n_i} \right)
    \]
for some scalars $\alpha_1,\ldots, \alpha_d\geq 0$ satisfying $\sum_{i=1}^d \alpha_i=1$. Hence $\sigma(B,\phi) = 1 - \sum_{i=1}^d \frac{\alpha_i^2}{n_i^2}$, and using Theorems~\ref{thm:direct_sum_formula} and \ref{thm:amplification_formula} we have
    \begin{align*}
        \sigma(A\otimes B,\tau\otimes\phi) &= \sum_{i=1}^d \alpha_i^2 \sigma(A\otimes M_{n_i}(\C),\tau\otimes\tr_{n_i}) + \alpha_i(1-\alpha_i)\\
            &= \sum_{i=1}^d \alpha_i^2+\frac{\alpha_i^2}{n_i^2}( \sigma(A,\tau)-1) + \alpha_i -\alpha_i^2 \\
            &= (1-\sigma(B,\phi))(\sigma(A,\tau)-1) +1= \sigma(A,\tau)+\sigma(B,\phi)-\sigma(A,\tau)\sigma(B,\phi).\qedhere
    \end{align*}
\end{proof}

\section{Algebraic Relations}\label{sec:algebraic_relations}

In this section, we explore how algebraic relations among the generators of $A$ can yield useful upper bounds for the free Stein dimension.
In some cases these bounds are dependent on the size of the generating set, and so the results will be stated in terms of generating sets $X=X^*\subset A$ rather than $A$ itself.

We begin with a technical lemma, which in essence captures the fact that algebraic relations in $B\<X\>$ must have derivatives orthogonal to the images of the generators under any derivation in $\Der_{1\otimes 1}(B\subset B\<X\>,\tau)$.
\begin{lem}\label{lem:weak_convergence_upper_bound}
Let $(M,\tau)$ be a tracial von Neumann algebra with unital $*$-subalgebra $B\subset M$ and finite subset $X=X^*\subset M$ so that $B\<X\>''=M$. Suppose $(P^{(i)})_{i\in I} =( p_1^{(i)},\ldots, p_d^{(i)})_{i\in I}\subset B\<T_X\>^d$ is a net such that $(P^{(i)}(X), (\J_{X:B} P^{(i)})(X))_{i\in I}$ converges weakly to $(0,H) \in L^2( M)^{\oplus d}\oplus M_{d\times |X|}(L^2(M\bar\otimes M^\circ,\tau\otimes \tau^\circ))$. Let $\H$ be the closure of 
    \[
        \left\{\left( \sum_{j=1}^d J_\ttop m_j[H]_{j,x}\right)_{x\in X}\colon m_1,\ldots, m_d\in \mmop \right\} \subset L^2(\mmop,\tau\otimes\tau^\circ)^{X}.
    \]
Then $\H_\mmop \leq L^2(\mmop,\tau\otimes\tau^\circ)^{X}_{\mmop}$ (under the diagonal action) and $\sigma(B\subset B\<X\>, \tau) \leq \dim \H^\perp_\mmop$.
\end{lem} 
\begin{proof}
Note that the assumed convergence implies $(\partial_{x : B} p_j^{(i)})(X)\to [H]_{j,x}$ weakly for each $j=1,\ldots, d$ and $x\in X$. Now, for $\delta\in \Der_{1\otimes 1}(B\subset B\<X\>,\tau)$ and $(m_1,\ldots, m_d)\in (B\<X\>\otimes B\<X\>^\circ)^d$ we have
	\begin{align*}
	    \sum_{x\in X} \< \delta(x),  \sum_{j=1}^d J_\ttop m_j [H]_{j,x}\>_\ttop &= \lim_{i\to\infty} \sum_{x\in X} \< \delta(x),  \sum_{j=1}^d J_{\ttop} m_j (\partial_{x : B} p_j^{(i)})(X)\>_\ttop\\
			& = \lim_{i\to\infty}  \sum_{j=1} \sum_{x\in X} \<  (\partial_{x : B} p_j^{(i)})(X)\delta(x), m_j^*\>_\ttop\\
			& = \lim_{i\to\infty}  \sum_{j=1} \<  \delta(p_j^{(i)}(X)), m_j^*\>_\ttop\\
			&=\lim_{i\to\infty}  \sum_{j=1} \<  p_j^{(i)}(X), \delta^*(m_j^*)\>_\ttop=0.
	\end{align*}
Thus with $\phi_X$ as in Lemma~\ref{lem:derivations_as_Hilbert_spaces}, we have
	\[
	    \phi_X(\Der_{1\otimes 1}(B\subset B\<X\>,\tau))_{\mmop} \leq \H^\perp_{\mmop},
	\]
and the result follows.
\end{proof}

Observe that if $H$ in the previous theorem belongs to $M_{d\times |X|}(M\bar\otimes M^\circ)$, then for $m_1,\ldots, m_d\in \mmop$ and $x\in X$ we have
    \[
        \left(\sum_{j=1}^d J_\ttop m_j[H]_{j,x}\right)_{x\in X} = \left(\sum_{j=1}^d [H]_{j,x}^* m_j^* = \sum_{j=1}^d [H^*]_{x,j} m_j^* \right)_{x\in X} = H^* (m_1^*,\ldots, m_d^*)^T.
    \]
Thus $\H=\overline{\text{ran}{H^*}}$ and $\H^\perp = \ker{\H}$. These observations lead to the following theorem.

\begin{thm}\label{thm:upper_bound_by_algebraic_relations}
Let $(M,\tau)$ be a tracial von Neumann algebra with unital $*$-subalgebra $B\subset M$ and finite subset $X=X^*\subset M$ so that $B\<X\>''=M$. Suppose $(P^{(i)})_{i\in I}\subset B\<T_X\>^d$ is a net such that $(P^{(i)}(X))_{i\in I}$ converges weakly to zero while $( ( \J_{X:B} P^{(i)})(X))_{i\in I}$ converges coordinate-wise in the weak operator topology to some $H\in M_{d\times |X|}(M\bar\otimes M^\circ)$. Then
	\[
		\sigma(B\subset B\<X\>,\tau) \leq \dim (\ker{H})_{M\bar\otimes M^\circ},
	\]
where $\ker(H)_{\mmop}\leq L^2(\mmop,\ttop)^{X}_{\mmop}$ under diagonal action. In particular, if $P(X)=0$ for some $P\in B\<T_X\>^d$, then
    \[
        \sigma(B\subset B\<X\>,\tau) \leq \dim (\ker  (\J_{X:B} P)(X))_{\mmop}.
    \]
\end{thm}
\begin{proof}
The assumed convergence implies $(P^{(i)}, ( \J_{X:B} P^{(i)})(X))_{i\in I}$ converges weakly to $(0, H)$. The first inequality then follows from the discussion preceding the theorem. The second inequality is simply the special case where the net is constant.
\end{proof}

Applying this to commutators leads us to the following corollary.
\begin{cor}\label{cor:diffuse_center}
Let $(M,\tau)$ be a tracial von Neumann algebra with a finitely generated unital $*$-subalgebra $A\subset M$. If $A$ has a diffuse central element, then $\sigma(A,\tau)=1$.
\end{cor}
\begin{proof}
Since $A''$ contains a diffuse element, we have $\sigma(A,\tau)\geq 1$ by \cite[Theorem 3.8]{CN21}. Let $X=X^*\subset A$ be a generating set which contains a diffuse central element $x_0$. Consider $P:=(t_{x_0}t_x - t_x t_{x_0})_{x\neq x_0}$. Then $P(X)=0$ by the centrality of $x_0$, and so by Theorem~\ref{thm:upper_bound_by_algebraic_relations} it suffices to show $\dim (\ker (\J_{X}P)(X))_{\mmop}\leq 1$. We will show $( \J_{X}P)(X)$ has dense range, and so the needed bound will therefore follow by the rank-nullity theorem. Given any $\eta\in L^2(\mmop,\ttop)^{X\setminus \{x_0\}}$ and $\epsilon>0$, and define $\xi\in L^2(\mmop,\ttop)^X$ by
    \[
        \xi_x:=\begin{cases} 0 & \text{if }x=x_0\\ [x_0,1\otimes 1]^{-1} 1_{|[x_0,1\otimes 1]|>\epsilon} \eta_x & \text{otherwise} \end{cases}.
    \]
Then $(\J_X P)(X) \xi = 1_{|[x_0,1\otimes 1]|>\epsilon} \eta$. Since $x_0$ is diffuse, $\ker([x_0,1\otimes 1])=\{0\}$ and so $1_{|[x_0,1\otimes 1]|>\epsilon}$ converges strongly to $1$ as $\epsilon\to 0$. Therefore $(\J_X P)(X)$ has dense range.
\end{proof}

\begin{ex}
For $N\geq 3$, let $(\mc{L}(\mathbb{F}O_N),\tau)$ be an orthogonal free quantum group factor with its unique trace and let $U=\{u_{i,j}\colon 1\leq i,j\leq N\}$ be the standard generators.
If $F$ is the tuple of relations induced by $\sum_{i,j=1}^N u_{i,j}\otimes E_{i,j}$ being a unitary matrix, then
	\[
		\dim \ker(\J_U F)_{\mc{L}(\mathbb{F}O_N)\bar\otimes \mc{L}(\mathbb{F}O_N)^\circ} = N^2 - \text{rank} (\J_U F) = 1,
	\]
by \cite[Lemma 4.1]{BV18}. Hence $\sigma(U,\tau)\leq 1$ by Theorem~\ref{thm:upper_bound_by_algebraic_relations}. On the other hand, $\mc{L}(\mathbb{F}O_N)$ is diffuse and so $\sigma(U,\tau)=1$.$\hfill\blacksquare$
\end{ex}

\begin{rem}
For a tracial von Neumann algebra $(M,\tau)$ and $X:=\{x_1,\ldots, x_n\}\subset M_{s.a.}$, one always has
    \[
        \sigma(\C\<X\>,\tau) \leq \delta^*(x_1,\ldots, x_n) \leq \delta^\star(x_1,\ldots, x_n) \leq \Delta(x_1,\ldots, x_n) 
    \]
(see \cite[Corollary 4.4]{CN21} and \cite[Lemma 4.1 and Theorem 4.4]{CS05}). If $x_1$ is diffuse and commutes with the other elements of the tuple, then $\sigma(\C\<X\>,\tau)=1$ by the previous corollary. Moreover, $\Delta(x_1,\ldots, x_n)=1$ by \cite[Lemma 3.4]{CS05}. Consequently, $\delta^*(x_1,\ldots, x_n) = \delta^\star(x_1,\ldots, x_n)=1$ in this case as well. If $\C\<X\>''$ embeds into the ultrapower of the hyperfinite $\II_1$ factor $\mathcal{R}^\omega$, then this can be deduced from \cite{BCG03} and \cite[Corollary 4.7]{Jun03}. However, our arguments also hold without this assumption.
\end{rem}

Recall that when the free Stein dimension is full (i.e. $\sigma(\C\<X\>,\tau)=|X|$), then $\partial_X$ and $\J_X$ are closable operators (see \cite[Corollary 4.7]{CN21}).

\begin{thm}\label{thm:maximal_fSd}
Let $(M,\tau)$ be a tracial von Neumann algebra with finite subset $X=X^*\subset M$. If $\sigma(\C\<X\>,\tau)=|X|$, then
    \[
        \sigma(\C\<X\cup Y\>,\tau) \leq |X| = \sigma(\C\<X\>,\tau).
    \]
for any finite subset $Y=Y^*\subset \dom(\overline{\partial_X})\cap \C\<X\>''$.
\end{thm}
\begin{proof} 
We may assume that $M=\C\<X\>''$. Since $Y\subset \dom(\overline{\partial_X})$, for each $y\in Y$ we can find a sequence $(p^{(y)}_n)_{n\in \N}\subset \C\<X\>$ so that $p^{(y)}_n\to y$ in $L^2(M,\tau)$ while $\partial_X p^{(y)}_n\to \overline{\partial_X}(y)$ in $L^2(\mmop,\ttop)$. Define $Q_n:=( y - p^{(y)}_n)_{y\in Y}\in \C\<X\cup Y\>^Y$. It follows that $(Q_n, \J_{X\cup Y} Q_n)\to (0,H)$ weakly, where
    \[
        H = (-\overline{\J_X}Y \mid \1 ).
    \]
If $\H$ is as in Lemma~\ref{lem:weak_convergence_upper_bound}, then $\H$ is the right Hilbert $\mmop$-module generated by the rows of $H$. Consider the bounded right $\mmop$-linear map $\H\ni (\xi_z)_{z\in X\cup Y}\mapsto (\xi_y)_{y\in Y}\in L^2(\mmop,\ttop)^Y$. Clearly this has dense range and consequently $\dim \H_{\mmop} \geq |Y|$. The desired inequality then follows from Lemma~\ref{lem:weak_convergence_upper_bound}.
\end{proof}

\begin{thm}\label{thm:including_a_unitary}
Let $(M,\tau)$ be a tracial von Neumann algebra and $A\subset M$ a finitely generated unital $*$-subalgebra containing a diffuse element $x$. Suppose $u\in M$ is a unitary such $uxu^*\in A$. If $B$ is the $*$-algebra generated by $A\cup\{u\}$ then $\sigma(B,\tau)\leq \sigma(A,\tau)$.
\end{thm}
\begin{proof}
We claim that
	\[
		\Der(B,\tau)\ni \delta\mapsto \delta|_A\in \Der(A,\tau)
	\]
is an injection. Suppose $\delta|_A\equiv 0$. Then we have
	\[
		0 = \delta( ux - (uxu^*)u) = \delta(u)\cdot x - (uxu^*)\cdot \delta(u) = (1\otimes x - uxu^*\otimes 1)\# \delta(u).
	\]
Since $x$ is diffuse, $1\otimes x - uxu^*\otimes 1$ has trivial kernel and so we must have $\delta(u)=0$. But then $\delta\equiv 0$.
\end{proof}

We cannot expect equality in the above theorem: let $x$ be a semicicular operator and let $u$ be a Haar unitary that is free from $x$. Set $y=uxu^*$ and $z=u^2 x(u^*)^2$, and let $A=\C\<x,y,z\>$. Since these are freely independent semicircular operators, we have $\sigma(A)=3$. On the other hand, the algebra generated $A\cup{u}$ is $\C\<x,u\>$ which has free Stein dimension two. Note, however that $uAu^*\neq A$ since $uzu^*\not\in A$. 

Even if one assumes $u$ normalizes all of $A$, one should still not expect equality: for $M=A\otimes (\C\oplus \C)$ and let $u=1\otimes (1\oplus -1)$ for $A$ with $\sigma(A) > 1$, Theorem~\ref{thm:direct_sum_formula} gives $\sigma(M) = \frac12+\frac12\sigma(A) < \sigma(A)$.

\begin{ex}\footnote{We are grateful to Vaughan Jones for suggesting this example.}
For $\theta\in \R\setminus \Q$, let $\mathcal{A}_\theta$ be the irrational rotation algebra with unique trace $\tau$ and unitary generators $vu=e^{2\pi i\theta}uv$. Then $u$ and $v$ are Haar unitaries, and so $A:=\C\<u,u^*\>$ contains the diffuse element $u+u^*$ and
    \[
        v(u+u^*)v^* = e^{2\pi i \theta} u + e^{-2\pi i \theta} u^* \in A.
    \]
So by Theorem~\ref{thm:including_a_unitary} and \cite[Theorem 3.8]{CN21}
    \[
        1 \leq \sigma(\C\<u,u^*,v,v^*\>, \tau) \leq \sigma(\C\<u,u\>,\tau) \leq 1.
    \]
Thus $\sigma(\C\<u,u^*,v,v^*\>,\tau)=1$. In particular, this quantity is indepenedent of $\theta$. This suggests the free Stein dimension is not a good invariant for $C^*$-algebras. $\hfill\blacksquare$
\end{ex}

\begin{rem}
Let $\mathcal{R}$ be the hyperfinite $\II_1$ factor with its unique trace $\tau$.
Heuristically, $\mathcal{R}$ is a limit of matrix algebras and
    \[
        \sigma(M_n(\C),\tr_n) = 1- \frac{1}{n^2}
    \]
for all $n\in \N$ by \cite[Corollary 5.2]{CN21}, so one should suspect that $\sigma(A,\tau)=1$ for any weakly dense finitely generated unital $*$-subalgebra $A\subset \mathcal{R}$. This is supported by the previous example since the von Neumann algebra generated by $\mathcal{A}_\theta$ under the GNS representation with respect to its unique trace is $\mathcal{R}$. Additionally, if $\Gamma$ is an i.c.c. amenable group so that $L(\Gamma)=\mathcal{R}$, then
    \[
        \sigma(\C[\Gamma],\tau)=\beta_1^{(2)}(\Gamma) - \beta_0^{(2)}(\Gamma)+1=1
    \]
by \cite[Proposition 5.1]{CN21}. In general, one always has $\sigma(A,\tau) \geq 1$ by \cite[Theorem 3.8]{CN21} since $\mathcal{R}$ is diffuse. So to prove this conjecture one need only establish the other inequality, but this seems to be beyond the reach of our current techniques.
\end{rem}

\section{Separable Abelian von Neumann Algebras}\label{sec:abelian_vNas}
In this section, we turn our attention to the abelian case.
We will show in Theorem~\ref{thm:fSd_for_abelian_vNas} that if $(M, \tau)$ is a separable abelian tracial von Neumann algebra, then the free Stein dimension of any weakly dense finitely generated unital $*$-subalgebra of $M$ depends only on the traces of the minimal projections in $M$. In particular, the free Stein dimension in this case is an invariant of the von Neumann algebra (and the choice of trace) rather than just an invariant of the $*$-subalgebra.
First, however, we will need two technical lemmas.

\begin{lem}
\label{lem:minimal_central_projections_are_orthogonal_to_derivations}
Let $(M,\tau)$ be a tracial von Neumann algebra with finitely generated unital $*$-subalgebra $A\subset M$.
Suppose $p\in A''$ is a central projection that is minimal in $A''$.
Then for any closable derivation $\delta\colon A\to L^2(\aaop,\ttop)$ one has $\delta(A)\perp p\otimes p$.
\end{lem}

\begin{proof}
	Suppose $\delta \colon A \to L^2(\aaop, \ttop)$ is closable, $p \in A''$ is central and minimal in $A''$, and $x \in A$.
	Our goal is to show $\delta(x) \perp p\otimes p$; note that it suffices to show this for $x = x^*$.
	By the minimality of $p$ we have that $px$ is a scalar multiple of $p$; translating $x$ by a scalar, we may assume $px = 0$.

	Let $f_n$ be a sequence of polynomials converging uniformly to zero on $[-\norm{x}, \norm{x}]$ so that $f_n'(0) = 1$, $f_n'$ converges uniformly to zero on compact subsets of $[-\norm{x}, \norm{x}] \setminus \set0$, and $f_n'$ are uniformly bounded.
	(For example, start with a sequence of polynomials uniformly approximating increasingly narrow spikes at $0$, and take a sequence of antiderivatives all of which vanish at $0$.)
	We have by choice of $f_n$ that $f_n(x) \to 0$ in norm and therefore in $L^2(A, \tau)$.

	Now, since $f_n$ is a polynomial, we have the algebraic identity
	\[\delta(f_n(x)) = \paren{\frac{f_n(x)\otimes 1 - 1\otimes f_n(x)}{x\otimes1 - 1\otimes x}}\cdot\delta(x).\]
	We claim that this sequence converges in $L^2(\aaop, \ttop)$.

	Notice that the difference quotients
	\[d_n : (s, t) \mapsto \frac{f_n(s)  - f_n(t)}{s - t}\]
	extend continuously to the diagonal, and are uniformly bounded, say by a constant $C$.
	(Uniform boundedness follows from the fact that $f_n'$ are uniformly bounded and the Intermediate Value Theorem.)
	Moreover, on any compact subset of $[-\norm x, \norm x]^2 \setminus\set{(0,0)}$, they tend uniformly to $0$.

	Let $\epsilon > 0$, and write $1 = p_0 + p_\epsilon + q_\epsilon$ where $p_0, p_\epsilon, q_\epsilon$ are the spectral projections of $x$ onto $\set{0}$, $(-\epsilon, \epsilon)\setminus\set0$, and $\R\setminus(-\epsilon, \epsilon)$ respectively.
	The above convergence of the difference quotients implies that for each $\epsilon > 0$,
	\[(q_\epsilon \otimes 1) \delta(f_n(x)), (1\otimes q_\epsilon)\delta(f_n(x)) \to 0\]
	as $n\to\infty$. Since $p_\epsilon \to 0$ in the $\sigma$-strong operator topology,
	    \[
	        \|(p_\epsilon\otimes 1)\delta(f_n(x))\|_{\ttop} = \|d_n(x) (p_\epsilon \otimes 1)\delta(x)\|_{\ttop} \leq C \| (p_\epsilon\otimes 1)\delta(x)\|_{\ttop} \to 0
	    \]
	as $\epsilon \to 0$. Similarly for $(1\otimes p_\epsilon)\delta(f_n(x))$.
	We conclude that $\delta(f_n(x)) \to (p_0\otimes p_0)\delta(x)$ in $L^2(\aaop, \ttop)$: indeed, given $\rho > 0$, choose $\epsilon > 0$ so that $\norm{(p_\epsilon\otimes 1)\delta(f_n(x))}_{\ttop}$ and $\norm{(1\otimes p_\epsilon)\delta(f_n(x))}_\ttop$ are less than $\frac\rho2$ for all $n$, and then for $n$ sufficiently large $\norm{(q_\epsilon\otimes1)\delta(f_n(x))}_\ttop$ and $\norm{(1\otimes q_\epsilon)\delta(f_n(x))}_\ttop$ are less than $\frac\rho2$, too.

	Now, since $\delta$ is closable, we must have $\delta(f_n(x)) \to \delta(0) = 0$.
	Therefore $(p_0\otimes p_0) \delta(x) = (p_0\otimes p_0)\delta(f_n(x)) \to 0$, so $(p_0\otimes p_0) \delta(x) = 0$.
	Finally, since $pp_0 = p$, $(p\otimes p)\delta(x) = (p\otimes p)(p_0 \otimes p_0)\delta(x) = 0,$ as desired.
\end{proof}


%

\begin{lem}\label{lem:existence_of_diffuse_element}
	Suppose that $(M, \tau)$ is an abelian von Neumann algebra and $A \subseteq M$ is a weakly dense finitely generated unital $*$-subalgebra.
	If $M$ contains a diffuse element, then so does $A$.
\end{lem}

\begin{proof}
	Suppose that $A$ is generated by $x_1, \ldots, x_n$ and contains no diffuse element.
	Then for every tuple $(\lambda_1, \ldots, \lambda_n) \in \R^n$, $\lambda_1 x_1 + \cdots + \lambda_n x_n$ has an atom of non-zero mass in its spectral measure.
	Now, take a finite collection of vectors $(\lambda_{1,i}, \ldots, \lambda_{n, i}) \in \R^n$ and constants $c_i \in \C$ so that any $n$ of the vectors are linearly independent, and the total mass of the atoms of $\lambda_{1,i}x_1 + \cdots + \lambda_{n, i}x_n$ at $c_i$ is strictly more than $n-1$.
	(To find such a collection, we just need to note the fact that there is an uncountable set of vectors in $\R^n$ with the property that all subsets of size $n$ are linearly independent, such as $\set{(1, 2^t, 3^t, \ldots, n^t) \mid t \in \R}$.)

	Since the total mass of the corresponding spectral projections is more than $n-1$, there must be a collection of $n$ of them with non-zero product; since we are in a commutative von Neumann algebra this product is in fact another projection, on which the relations
	\begin{align*}
		\lambda_{1,1}x_1 + \ldots + \lambda_{n,1}x_n &= c_1 \\
		\lambda_{1,2}x_1 + \ldots + \lambda_{n,2}x_n &= c_2 \\
		&\vdots \\
		\lambda_{1,n}x_1 + \ldots + \lambda_{n,n}x_n &= c_n \\
	\end{align*}
	all hold.
	Since the coefficients were chosen to be linearly independent, on this projection each $x_i$ must be constant, and hence this atom is present in every element of the von Neumann algebra they generate.
\end{proof}

\begin{thm}\label{thm:fSd_for_abelian_vNas}
Let $(M,\tau)$ be a separable abelian tracial von Neumann algebra with set of minimal projections $P$. Then for any weakly dense finitely generated unital $*$-subalgebra $A\subset M$ one has
    \[
        \sigma(A,\tau) = 1 - \sum_{p\in P} \tau(p)^2.
    \]
\end{thm}
\begin{proof}
For $x\in A$ and each $p\in P$, let $x_p\in \C$ denote the scalar such that $xp=x_pp$. Note that $P$ is countable since $\tau$ is faithful. We will show every closable derivation $\delta\colon A\to L^2(\aaop, \ttop)$ is approximately inner, so that in particular one has
     \[
        \sigma(A,\tau) = \dim \overline{\InnDer(A,\tau)}_{\mmop} = \dim L^2_A(\aaop,\ttop)^\perp_{\mmop} = 1 - \sum_{p\in P} \tau(p)^2.
    \]
by Lemma~\ref{lem:derivations_from_non-A-central_vectors}.
Denote $p_0:=1- \sum_{p\in P}p$ and observe that by Lemma~\ref{lem:minimal_central_projections_are_orthogonal_to_derivations}, we have
    \[
        \delta = \delta\cdot\left(1\otimes 1 - \sum_{p\in P} p\otimes p\right) = \delta\cdot(p_0\otimes p_0) + \sum_{p\in P} \delta\cdot(p_0\otimes p + p\otimes p_0) + \sum_{\substack{p,q\in P\\ p\neq q }} \delta\cdot(p\otimes q).
    \]
Since $P$ is countable, it suffices to show each of the derivations in the above sum is approximately inner.

As $Mp_0$ is diffuse, Lemma~\ref{lem:existence_of_diffuse_element} implies that $Ap_0$ contains a self-adjoint diffuse element, say $x_0p_0$ for $x_0\in A$. For $\epsilon>0$, define
    \[
        \xi_\epsilon:= 1_{|[x_0,p_0\otimes p_0]|>\epsilon}[x_0,p_0\otimes p_0]^{-1} \delta(x_0)(p_0\otimes p_0),
    \]
where the inverse is as an operator on $L^2(\aaop,\ttop)(p_0\otimes p_0)$. Now, for any $x\in A$, $[x_0,x]=0$ implies $[x_0,\delta(x)]=[x,\delta(x_0)]$. Multiplying by $p_0\otimes p_0$ yields
    \begin{align*}
        [x,\xi_\epsilon] &= 1_{|[x_0,p_0\otimes p_0]|>\epsilon}[x_0,p_0\otimes p_0]^{-1}[x,\delta(x_0)](p_0\otimes p_0)\\
            &= 1_{|[x_0,p_0\otimes p_0]|>\epsilon}[x_0,p_0\otimes p_0]^{-1}[x_0,\delta(x)](p_0\otimes p_0)\\
            &= 1_{|[x_0,p_0\otimes p_0]|>\epsilon}\delta(x)(p_0\otimes p_0).
    \end{align*}
Since $x_0p_0$ is diffuse, $1_{|[x_0,p_0\otimes p_0]|>\epsilon}$ converges strongly to $p_0\otimes p_0$, and so $[\cdot,\xi_\epsilon]\to \delta\cdot (p_0\otimes p_0)$ pointwise on $A$.

Next, for $p\in P$ and $\epsilon>0$ define
    \[
        \eta_\epsilon:= \left(1_{|x_0p_0 - (x_0)_pp_0|>\epsilon} (x_0p_0 - (x_0)_pp_0)^{-1}\right)\cdot \delta(x_0)(p_0\otimes p).
    \]
Using $[x_0,\delta(x)]=[x,\delta(x_0)]$ and $(p_0\otimes p)\cdot x_0 = (x_0)_p p_0\otimes p$ gives
    \begin{align*}
        [x,\eta_\epsilon] &=  \left(1_{|x_0p_0 - (x_0)_pp_0|>\epsilon} (x_0p_0 - (x_0)_pp_0)^{-1}\right)\cdot [x,\delta(x_0)](p_0\otimes p)\\
            &= \left(1_{|x_0p_0 - (x_0)_pp_0|>\epsilon} (x_0p_0 - (x_0)_pp_0)^{-1}\right)\cdot [x_0,\delta(x)](p_0\otimes p)\\
            &= \left(1_{|x_0p_0 - (x_0)_pp_0|>\epsilon} (x_0p_0 - (x_0)_pp_0)^{-1}\right)(x_0p_0 - (x_0)_pp_0)\cdot\delta(x)(p_0\otimes p)\\
            &=1_{|x_0p_0 - (x_0)_pp_0|>\epsilon}\cdot \delta(x)(p_0\otimes p).
    \end{align*}
Since $x_0p_0$ is diffuse, $1_{|x_0p_0 - (x_0)_p p_0|>\epsilon}$ converges $\sigma$-strongly to $p_0$, and thus $[\cdot, \eta_\epsilon]\to \delta\cdot (p_0\otimes p)$ pointwise on $A$. We omit the proof of $\delta\cdot (p\otimes p_0)$, which is similar.

Finally, fix a pair of distinct projections $p,q\in P$, and let $a\in A$ be such that $a_p\neq a_q$. Such an element necessarily exists, since otherwise $p-q\not\in A''=M$. Define
    \[
        \zeta_{p,q}:= \frac{1}{a_p - a_q} \delta(a)(p\otimes q).
    \]
Using $[a,\delta(x)]= [x,\delta(a)]$ for $x\in A$ gives
    \[
        (a_p - a_q) \delta(x)(p\otimes q) = (x_p - x_q)\delta(a)(p\otimes q)=(x_p - x_q) (a_p - a_q) \zeta_{p,q} = (a_p -a_q) [x,\zeta_{p,q}].
    \]
Hence $\delta(x)(p\otimes q) = [x,\zeta_{p,q}]$ is inner.
\end{proof}



\begin{cor}
Let $\tau\colon \ell^\infty\to\C$ be a faithful normal tracial state, and let $A\subset \ell^\infty$ be a weakly dense finitely generated unital $*$-subalgebra. 
Then
    \[
        \sigma(A,\tau) = 1 - \sum_{n=1}^\infty \tau(e_n)^2.
    \] 
\end{cor}

\section{$L^2$-Rigidity}\label{sec:L^2_rigidity}

Let $(M,\tau)$ be a tracial von Neumann algebra. After \cite{DI16} we say a finite subset $S\subset M$ is a \emph{non-amenability set} for $M$ if there exists a constant $K>0$ such that
    \begin{align}\label{eqn:non-amenability_set}
        \|\xi\|_{\tau\otimes \tau^\circ} \leq K \sum_{x\in S} \| [x,\xi] \|_{\tau\otimes \tau^\circ} 
    \end{align}
for all $\xi\in L^2(\mmop, \tau\otimes\tau^\circ)$. It should be noted that this idea also appeared implicitly in the work of Connes, for example \cite[Theorem 5.1.6]{Con76}. The existence of such a set implies $M$ has no amenable direct summand, and by \cite{Con76} a $\II_1$ factor is non-amenable if and only if it has a non-amenability set.

If $(M,\tau)$ has a non-amenability set, then (\ref{eqn:non-amenability_set}) implies
    \begin{align}\label{eqn:vector_to_inner_derivation}
        L^2(A\otimes A^\circ,\tau\otimes \tau^\circ) \ni \xi\mapsto [\cdot,\xi]\in \InnDer(A,\tau)
    \end{align}
is injective for any finitely generated unital $*$-subalgebra $A\subset M$. Hence $L_A^2(A\otimes A^\circ,\tau\otimes\tau^\circ)=\{0\}$, and so Lemma~\ref{lem:derivations_from_non-A-central_vectors} implies
    \[
        \sigma(A,\tau) \geq \dim \overline{\InnDer(A,\tau)}_{(A\otimes A^\circ)''} = \dim L^2(A\otimes A^\circ)_{(A\otimes A^\circ)''}=1.
    \]
Thus if $\sigma(A,\tau)>1$, then $\Der_{1\otimes 1}(A,\tau)$ necessarily contains a non-inner derivation.

If one further assumes that $A$ contains a non-amenability set, then any finite generating set $X=X^*\subset A$ is a non-amenability set and (\ref{eqn:non-amenability_set}) implies the map in (\ref{eqn:vector_to_inner_derivation}) has a inverse bounded with respect to the $\|\cdot\|_X$-norm. In particular, $\InnDer(\C\<X\>,\tau)=\overline{\InnDer(\C\<X\>,\tau)}$; that is, all approximately inner derivations are actually inner. 

\begin{thm}\label{thm:L^2-rigid}
Let $M$ be a non-amenable $\II_1$ factor with trace $\tau$ and let $A\subset M$ be a weakly dense finitely generated unital $*$-subalgebra. If $A$ contains a non-amenability set and $\sigma(A,\tau)>1$, then $M$ is not $L^2$-rigid. In particular, $M$ is prime and does not have property Gamma.
\end{thm}
\begin{proof}
By the discussion preceding the theorem, $\sigma(A,\tau)>1$ implies that there exists $\delta\in \Der_{1\otimes 1}(A,\tau)$ which is not inner. Then, recalling the notation from Subsection~\ref{subsec:derivation_spaces}, at least one of the real derivations
    \[
        \text{Re}(\delta)= \frac{1}{2}(\delta + \delta^\dagger)\qquad \text{ or }\qquad \text{Im}(\delta)= \frac{1}{2i}(\delta - \delta^\dagger)
    \]
is not inner, and hence unbounded by \cite[Theorem 2.2]{Pet09}. The result then follows immediately from \cite[Theorem 1.1]{DI16}.
\end{proof}

Equivalently, the above theorem states that if $M$ is an $L^2$-rigid $\II_1$ factor, then $\sigma(A,\tau)=1$ for any weakly dense finitely generated unital $*$-subalgebra $A\subset M$ containing a non-amenability set for $M$.

\begin{rem}
The hypothesis in Theorem~\ref{thm:L^2-rigid} that $A$ contains a non-amenability set is only used to apply \cite[Theorem 1.1]{DI16}. Dabrowski and Ioana observed that this a necessary hypothesis (see \cite[Remark 4.2]{DI16}); however, their counterexample to the general case is that of an approximately inner derivation. The discussion preceding Theorem~\ref{thm:L^2-rigid} shows that $\sigma(A,\tau)>1$ implies $A$ actually admits non-approximately inner derivations. It would be interesting to see if this condition on the free Stein dimension is strong enough to remove the non-amenability set hypothesis.
\end{rem}

After \cite[Definition 3.1]{Pet09} we say a finite subset $S\subset M$ is a \emph{non-Gamma set} for $M$ if there exists a constant $K>0$ such that
    \[
        \|\xi\|_\tau \leq K \sum_{x\in S} \| [x,\xi] \|_2
    \]
for all $\xi\in L^2(M,\tau)\ominus \C$. Note that this definition was motivated by \cite[Remark 4.1.6]{Pop86}. A $\II_1$ factor $M$ lacks property Gamma if and only if it has a non-Gamma set. Moreover, it follows from \cite{Con76} that a self-adjoint non-Gamma set for $M$ is also a non-amenability set for $M$ (see \cite[Lemma 2.10]{DI16} for a proof). We therefore obtain the following proposition as an immediate corollary of Theorem~\ref{thm:L^2-rigid} (or rather the discussion immediately following it) since $\II_1$ factors with property (T) are $L^2$-rigid (see \cite[Remark 4.2]{Pet09}). Since a direct proof is available and brief, we have also included it here.

\begin{prop}
Let $M$ be a $\II_1$ factor with property (T) and trace $\tau$, and let $A\subset M$ be a weakly dense finitely generated unital $*$-subalgebra. If $A$ contains a non-Gamma set for $M$, then $\sigma(A,\tau)=1$.
\end{prop}
\begin{proof}
By \cite[Theorem 3.2]{Pet09}, every closable derivation on $A$ is inner. In particular, $\Der_{1\otimes 1}(A,\tau) = \InnDer(A,\tau)$. Also note that since $M$ is diffuse, there are no $A$-central vectors in $L^2(A\otimes A^\circ, \tau\otimes\tau^\circ)$. Thus Lemma~\ref{lem:derivations_from_non-A-central_vectors} implies
    \[
        \sigma(A,\tau) = \dim \overline{\InnDer(A,\tau)}_{(A\otimes A^\circ)''} = \dim L^2(A\otimes A^\circ, \tau\otimes \tau^\circ)_{(A\otimes A^\circ)''}=1.\qedhere
    \]
\end{proof}

\section{Upper and Lower Free Stein Dimensions}\label{sec:upper_lower_fSd}

By considering limits superior and inferior, we give two extensions of the free Stein dimension to $*$-algebras that are \emph{not} finitely generated.
This yields invariants for tracial non-commutative probability spaces that agree with the free Stein dimension when the $*$-algebra is finitely generated.
In general it is difficult to even show these invariants are finite, let alone compute them explicitly.
However for tracial von Neumann algebras, the results of previous sections allow us to do precisely that for a number of interesting examples.

\begin{defi}
Let $(M,\tau)$ be a tracial von Neumann algebra. We define the \textbf{upper free Stein dimension} of an inclusion of unital $*$-subalgebras $B\subset A\subset M$ as the quantity
    \[
        \overline{\sigma}(B\subset A,\tau):= \limsup_{X=X^*\subset A} \sigma(B\subset B\<X\>,\tau),
    \]
where the limit supremum is over the directed set of self-adjoint finite subsets $X$ of $A$, ordered by inclusion. Similarly, the \textbf{lower free Stein dimension} of $B\subset A$ is the quantity
    \[
        \underline{\sigma}(B\subset A,\tau):=\liminf_{X=X^*\subset A} \sigma(B\subset B\<X\>,\tau).
    \]
For $B=\C$, we simply write $\overline{\sigma}(A,\tau)$ and $\underline{\sigma}(A,\tau)$. If these two quantities agree, we write $\sigma(B\subset A,\tau)$ for their common value, which we call the \textbf{free Stein dimension} of $B\subset A$.
\end{defi}

If $A$ is finitely generated over $B$, then clearly $\sigma(B\subset A,\tau)=\overline{\sigma}(B\subset A,\tau)=\underline{\sigma}(B\subset A,\tau)$. On the other hand, if $A$ is infinitely generated over $B$ then either of $\overline{\sigma}(B\subset A,\tau)$ or $\underline{\sigma}(B\subset A,\tau)$ could be infinite. Nevertheless, both are invariants of the non-commutative probability space $(A,\tau)$ in the sense that they are fixed under trace-preserving $*$-isomorphisms. 

\begin{rem}\label{rem:starting_upper_lower_fSd_later}
For unital $*$-subalgebras $B\subset A\subset M$, fix any any unital $*$-subalgebra $B\subset A_0 \subset A$ which is finitely generated over $B$. It follows that
    \[
        \overline{\sigma}(B\subset A,\tau) = \limsup_{X=X^*\subset A} \sigma(B\subset A_0\<X\>,\tau),
    \]
and similarly for the lower free Stein dimension.
\end{rem}

For the remainder of the section we focus on tracial non-commutative probability spaces coming from tracial von Neumann algebras. We begin by extending some of the properties for the free Stein dimension to the upper and lower free Stein dimensions of tracial von Neumann algebras.

\begin{prop}\label{prop:upper_lower_fSd_properties}
Let $(M,\tau)$ be a tracial von Neumann algebra.
\begin{enumerate}[label=(\arabic*)]
    \item If $M$ contains a self-adjoint diffuse element, then $\underline{\sigma}(M,\tau)\geq 1$.
    
    \item For any finite dimensional unital $*$-subalgebra $B\subset M$, one has
        \[
            \overline{\sigma}(M,\tau) = \overline{\sigma}(B\subset M,\tau) + \sigma(B,\tau)
        \]
    and
        \[
            \underline{\sigma}(M,\tau) = \underline{\sigma}(B\subset M,\tau) + \sigma(B,\tau).
        \]
    
    \item Suppose
         \[
             (M,\tau)=\left(\bigoplus_{i=1}^d M_i,\sum_{i=1}^d \alpha_i \tau_i\right),
        \]
    for a family $\{ (M_i,\tau_i)\}_{i=1}^d$ of tracial von Neumann algebras and scalars $\alpha_1,\ldots, \alpha_d> 0$ satisfying $\sum_{i=1}^d \alpha_i =1$. Then
        \[
            \overline{\sigma}\left(\bigoplus_{i= 1}^d M_i, \sum_{i=1}^d \alpha_i \tau_i \right) \leq \sum_{i=1}^d \alpha_i^2 \overline{\sigma}(M_i,\tau_i) + \alpha_i (1-\alpha_i).
        \]
    and
        \[
            \underline{\sigma}\left(\bigoplus_{i= 1}^d M_i, \sum_{i=1}^d \alpha_i \tau_i \right) \geq \sum_{i=1}^d \alpha_i^2 \underline{\sigma}(M_i,\tau_i) + \alpha_i (1-\alpha_i).
        \]
    If $\sigma(M_i,\tau_i)$ exists for all but possibly one $i=1,\ldots, d$, then the above are equalities.
    
    \item For any finite dimensional tracial von Neumann algebra $(B,\phi)$, one has
        \[
            \overline{\sigma}(M\otimes B,\tau\otimes \phi) \leq  \overline{\sigma}(M,\tau)+ \sigma(B,\phi) - \overline{\sigma}(M,\tau) \sigma(B,\phi)
        \]
    and
        \[
            \underline{\sigma}(M\otimes B,\tau\otimes \phi) \geq  \underline{\sigma}(M,\tau)+ \sigma(B,\phi) - \underline{\sigma}(M,\tau) \sigma(B,\phi).
        \]
    If $B$ is a factor, then these are equalities.
\end{enumerate} 
\end{prop}
\begin{proof}{\color{white},}

\noindent
\textbf{(1):} Let $x\in M$ be self-adjoint and diffuse. If $X=\{x\}$, then by \cite[Theorem 3.8]{CN21}
    \[
        \underline{\sigma}(M)\geq \inf_{Y\supset X} \sigma(Y) \geq 1.
    \]\\
 
\noindent\textbf{(2):} This follows immediately from Remark~\ref{rem:starting_upper_lower_fSd_later} and Theorem~\ref{thm:excising_a_finite_dimensional_subalgebra}.

\noindent\textbf{(3):} Let $1_i\in M_j$ denote the unit for each $i=1,\ldots, d$. For any finitely generated $*$-subalgebra $A\subset \bigoplus_{i=1}^d M_i$ containing $\{1_i\}_{i=1}^d$, one has $A= \bigoplus_{i=1}^d A 1_i$. Thus, by Remark~\ref{rem:starting_upper_lower_fSd_later} and Theorem~\ref{thm:direct_sum_formula} we have
    \begin{align*}
        \overline{\sigma}(M,\tau) &= \limsup_{X=X^*\subset M} \sigma( \C\<X\cup\{1_i\}_{i=1}^d\>, \tau)\\
            &\leq \sum_{i=1}^d \alpha_i^2 \limsup_{X=X^*\subset M}\sigma( \C\<X\> 1_i, \tau_i) + \alpha_i (1-\alpha_i) =  \sum_{i=1}^d \alpha_i^2 \overline{\sigma}( M, \tau_i) + \alpha_i (1-\alpha_i).
    \end{align*}
The inequality above is a consequence of the subaddivity of the limit supremum. Thus if $\sigma(M_i,\tau_i)$ exists for all but possibly one $i=1,\ldots, d$, then we in fact have equality above. The proof for the lower free Stein dimension is similar.\\

\noindent\textbf{(4):} First assume $B=M_n(\C)$ is a factor with $\phi=\tr_n$ and let $\{e_{i,j}\colon 1\leq i,j\leq n\}\subset M_n(\C)$ be a family of matrix units. For any finitely generated $*$-subalgebra $A\subset M\otimes M_n(\C)$ containing $\{1\otimes e_{i,j}\colon 1\leq i,j\leq n\}$, one has $A=(1\otimes \Tr)(A)\otimes M_n(\C)$. Thus, by Remark~\ref{rem:starting_upper_lower_fSd_later} and Theorem~\ref{thm:amplification_formula} we have
    \begin{align*}
        \overline{\sigma}(M\otimes M_n(\C),\tau\otimes\tr_n) &= \limsup_{X=X^*\subset M\otimes M_n(\C)} \sigma( \C\< X\cup\{1\otimes e_{i,j}\colon 1\leq i,j\leq n\}\>, \tau\otimes \tr_n)\\
            &= 1+ \frac{1}{n^2}  \limsup_{X=X^*\subset M\otimes M_n(\C)} \sigma( (1\otimes \Tr)(\C\<X\>), \tau) - \frac{1}{n^2}\\
            &= \frac{1}{n^2}\overline{\sigma}(M,\tau) + (1- \frac{1}{n^2})\\
            &= \overline{\sigma}(M,\tau) + \sigma(B,\phi) - \overline{\sigma}(M,\tau)\sigma(B,\phi)
    \end{align*}
For the general case, $B$ is necessarily a multimatrix algebra. Using the previous part and the factor case, one proceeds as in the proof of Corollary~\ref{cor:tensor_against_finite_dimensional_formula}. The proof for the lower free Stein dimension is similar.
\end{proof}

\begin{cor}
Let $(M,\tau)$ be a tracial von Neumann algebra that admits a unital inclusion of the hyperfinite $\II_1$ factor $\mathcal{R}$. Then
    \[
        \overline{\sigma}(\mathcal{R}\subset M,\tau) \leq \overline{\sigma}(M,\tau) - 1
    \]
and
    \[
        \underline{\sigma}(\mathcal{R}\subset M,\tau) \leq \underline{\sigma}(M,\tau)  - 1.
    \]
\end{cor}
\begin{proof}
For each $n\in \N$ we can find a unital inclusion $M_n(\C)\subset \mathcal{R}\subset M$. So using Proposition~\ref{prop:upper_lower_fSd_properties}.(2) we have
    \[
        \overline{\sigma}(\mathcal{R}\subset M,\tau) \leq \overline{\sigma}(M_n(\C)\subset M,\tau) = \overline{\sigma}(M,\tau) - 1 + \frac{1}{n^2}.
    \]
Taking the limit as $n\to\infty$ yields the desired inequality. The proof for the lower free Stein dimension is identical.
\end{proof}

Using the results of the previous sections, we can show the free Stein dimension exists and compute it for certain tracial von Neumann algebras.

\begin{thm}\label{thm:fSd_is_one_if}
Let $(M,\tau)$ be a tracial von Neumann algebra. If any of the following hold, then $\sigma(M,\tau)=1$.
    \begin{enumerate}[label=(\arabic*)]
    \item $M$ has diffuse center.
    \item $M$ is a finitely generated $L^2$-rigid $\II_1$ factor.
    \item $M$ is a $\II_1$ factor with $\overline{\sigma}(M,\tau)<\infty$ and whose fundamental group contains an element of $\Q\setminus \{1\}$.
    \end{enumerate}
\end{thm}
\begin{proof}
All three conditions imply $\underline{\sigma}(M,\tau)\geq 1$ by Proposition~\ref{prop:upper_lower_fSd_properties}.(1), and so it suffices to show $\overline{\sigma}(M,\tau)\leq 1$.

\noindent \textbf{(1):} Let $x_0\in M\cap M'$ be a self-adjoint diffuse element. Then by Remark~\ref{rem:starting_upper_lower_fSd_later} and Corollary~\ref{cor:diffuse_center} one has
    \[
        \overline{\sigma}(M,\tau) = \limsup_{X=X^*\subset M} \sigma(\C\<X\cup\{x_0\}\>,\tau) = 1.
    \]\\

\noindent \textbf{(2):} Since $M$ is an $L^2$-rigid $\II_1$ factor, it is necessarily non-amenable. Let $X_0=X_0^*\subset M$ be a non-amenability set which generates $M$. By Remark~\ref{rem:starting_upper_lower_fSd_later} and Theorem~\ref{thm:L^2-rigid} we have
    \[
        \overline{\sigma}(M,\tau)= \limsup_{X=X^*\subset M}\sigma( \C\<X\cup X_0\>,\tau) = 1.
    \]\\

\noindent \textbf{(3):} Let $q\in \Q\setminus\{1\}$ be in the fundamental group of $M$. Without loss of generality, $0<q<1$, so that $q=\frac{d}{n}$ for $n\in \N$ and $d\in \{1,\ldots, n-1\}$. Let $p\in M$ be a projection with $\tau(p)=\frac{1}{n}$. Then $(M,\tau)\cong (pMp\otimes M_d(\C), \tau^{(p)}\otimes \tr_d)$ and so by Proposition~\ref{prop:upper_lower_fSd_properties}.(4) we have
    \[
        \overline{\sigma}(M,\tau) = \overline{\sigma}(pMp,\tau^{(p)}) + \left( 1- \frac{1}{d^2} \right) - \left( 1 - \frac{1}{d^2}\right) \overline{\sigma}(pMp,\tau^{(p)}) = \frac{1}{d^2}\overline{\sigma}(pMp,\tau^{(p)}) + 1 - \frac{1}{d^2}.
    \]
Hence $\overline{\sigma}(pMp,\tau^{(p)}) = 1 + d^2(\overline{\sigma}(M,\tau) -1)$. Since $\tau(p)=\frac{1}{n}$ implies $(M,\tau)\cong (pMp\otimes M_n(\C),\tau^{(p)}\otimes \tr_n)$, the same computation implies we also have $\overline{\sigma}(pMp,\tau^{(p)}) = 1 + n^2(\overline{\sigma}(M,\tau) -1)$. Since $\overline{\sigma}(M,\tau)<\infty$ and $d\neq n$, we must therefore have $\overline{\sigma}(M,\tau)=1$.
\end{proof}

Condition $(1)$ of the previous theorem implies $\sigma(M,\tau)=1$ whenever $M$ is a diffuse abelian von Neumann algebra.
We can also readily extend Theorem~\ref{thm:fSd_for_abelian_vNas} to show that the free Stein dimension exists for any abelian von Neuman algebra, even those which are not separable:

\begin{thm}\label{thm:extended_fSd_abelian_vNas}
Let $(M,\tau)$ be an abelian tracial von Neumann algebra. Then
    \[
        \sigma(M,\tau) = 1 - \sum_{p} \tau(p)^2
    \]
where the sum is over the minimal projections of $M$. In particular, for a probability space $(X,\Omega,\mu)$ with atoms $\mathcal{A}\subset \Omega$ one has
    \[
        \sigma(L^\infty(X,\mu), \mathbb{E}) = 1 - \sum_{A\in\mathcal{A}} \mu(A)^2.
    \]
\end{thm}
\begin{proof}
We will first treat the case that $M$ is separable.
Let $A \subseteq M$ be a $*$-subalgebra which is finitely generated, weakly dense, and unital.
By Theorem~\ref{thm:fSd_for_abelian_vNas} we have for any finite self-adjoint tuple $X \subset M$ that
\[\sigma(A\ang{X}, \tau) = 1-\sum_{p}\tau(p)^2;\]
then by Remark~\ref{rem:starting_upper_lower_fSd_later} we conclude
\[\overline{\sigma}(M, \tau) = \limsup_{X = X^*\subset M} \sigma(A\ang{X}, \tau) = 1-\sum_{p}\tau(p)^2 = \liminf_{X = X^*\subset M} \sigma(A\ang{X}, \tau) = \underline{\sigma}(M, \tau).\]

If $M$ is not separable, we note that it nonetheless contains only countably many minimal projections.
We may then realize $M$ as the direct sum $M_a \oplus M_d$ where $M_a$ is separable and contains all the minimal projections, while $M_d$ is diffuse; let $\alpha$ be so that $\tau = \alpha\tau|_{M_a} \oplus (1-\alpha)\tau|_{M_d}$ in this decomposition.
Then the upper (respectively, lower) free Stein dimension of $M$ is bounded above (respectively, below) by
\[
\alpha^2\sigma(M_a, \tau_a) + (1-\alpha)^2\sigma(M_d, \tau_d) + 2\alpha(1-\alpha)
= \alpha^2\paren{1 - \sum_{p} \tau_a(p)^2} + (1-\alpha)^2 1 + 2\alpha(1-\alpha)
= 1 - \sum_p \tau(p)^2,
\]
by Proposition~\ref{prop:upper_lower_fSd_properties}.
(Here we have used the fact that $\sigma(M_a, \tau_a)$ exists with the correct value by the first part of this argument, while $\sigma(M_d, \tau_d) = 1$ by Theorem~\ref{thm:fSd_is_one_if}.)
The claim follows.
\end{proof}

We also have the following twist on Theorem~\ref{thm:fSd_is_one_if}.(3):

\begin{prop}
Suppose that $(M, \tau)$ is a $\II_1$ factor whose fundamental group contains an element of $\Q\setminus\set1$.
If $\underline{\sigma}(M, \tau) > 1$ then $\sigma(M, \tau) = \infty$.
\end{prop}

Let $(M,\tau)$ be a tracial von Neumann algebra with unital $*$-subalgebras $B\subset A\subset M$ such that $A$ is finitely generated over $B$. We conclude the paper by showing that the free Stein dimension over $B$ of a certain Dirichlet algebra containing $A$ (see \cite[Section 5]{DL92}) exists and equals $\sigma(B\subset A,\tau)$ (see Theorem~\ref{thm:fSd_of_Dirichlet_extension}). The idea is to define this Dirichlet algebra in a way that is ``uniform'' with respect to all $\delta\in \Der_{1\otimes 1}(B\subset A,\tau)$. To begin, we make a few observations regarding $\Der_{1\otimes 1}(B\subset A,\tau)$.

For any pair $\delta_1,\delta_2\in \Der_{1\otimes 1}(B\subset A,\tau)$, the map
    \begin{align*}
        \delta_1\oplus \delta_2\colon A &\to L^2(A,\tau)^2\\
            x&\mapsto (\delta_1(x), \delta_2(x))
    \end{align*}
is a closable derivation with $\dom(\overline{\delta_1\oplus \delta_2})\subset \dom(\overline{\delta_1})\cap \dom(\overline{\delta_2})$. In particular, $\text{Re}(\delta)\oplus \text{Im}(\delta)$ is a closable derivation for any $\delta\in \Der_{1\otimes 1}(A,\tau)$ with
    \[
        \dom( \overline{\delta\oplus \delta^\dagger}) = \dom(\overline{\text{Re}(\delta)\oplus \text{Im}(\delta)}).
    \]
For any $\xi$ in this common domain one has
    \begin{align}\label{eqn:closures_of_real_and_imaginary_derivations}
        \overline{\delta}(\xi)&= \overline{\text{Re}(\delta)}(\xi) + i \overline{\text{Im}(\delta)}(\xi) & \overline{\delta^\dagger}(\xi) &= \overline{\text{Re}(\delta)}(\xi) - i \overline{\text{Im}(\delta)}(\xi) \notag\\
        \overline{\text{Re}(\delta)}(\xi) &= \frac{1}{2}(\overline{\delta}(\xi) + \overline{\delta^\dagger}(\xi)) & \overline{\text{Im}(\delta)}(\xi) &= \frac{1}{2i}(\overline{\delta}(\xi) - \overline{\delta^\dagger}(\xi))
    \end{align}
Moreover, $\text{Re}(\delta)\oplus \text{Im}(\delta)$ is real:
    \[
        \<x\cdot (\text{Re}(\delta)\oplus \text{Im}(\delta))(y), (\text{Re}(\delta)\oplus \text{Im}(\delta))(z) \> = \<(\text{Re}(\delta)\oplus \text{Im}(\delta))(z^*), (\text{Re}(\delta)\oplus \text{Im}(\delta))(y^*)\cdot x^*\>
    \]
for all $x,y,z\in A$.

\begin{prop}\label{prop:Dirichlet_extension_of_A}
Let $(M,\tau)$ be a tracial von Neumann algebra with unital $*$-subalgebras $B\subset A\subset M$ such that $A$ is finitely generated over $B$. Consider
    \[
        \mathcal{D}:=\left\{y\in \bigcap_{\delta\in \Der_{1\otimes 1}(B\subset A,\tau)} A''\cap \dom(\overline{\delta \oplus \delta^\dagger}) \colon \sup_{\substack{\delta\in \Der_{1\otimes 1}(B\subset A,\tau)\\ \delta\neq 0}} \frac{\|\overline{\delta}(y)\|}{\|\delta\|_X} <\infty \right\},
    \]
where $X=X^*\subset A$ satisfies $B\<X\>=A$.

    \begin{enumerate}[label=(\arabic*)]
        \item $\mathcal{D}$ is a unital $*$-algebra containing $A$ and is independent of $X$.
        
        \item For any self-adjoint $y\in \mathcal{D}$ and Lipschitz function $f\colon \R\to \R$, one has $f(y)\in \mathcal{D}$.
        
        \item For any $y\in \mathcal{D}$, one has $|y|\in \mathcal{D}$.
        
        \item For any $c>0$,
            \[
                \left\{y\in \mathcal{D}\colon \sup_{\substack{\delta\in \Der_{1\otimes 1}(B\subset A,\tau)\\ \delta\neq 0}} \frac{\|\overline{\delta}(y)\|}{\|\delta\|_X}\leq c\right\}.
            \]
        is convex and closed in the weak operator topology.

    \end{enumerate}
\end{prop}
\begin{proof}
By the discussion preceding the proposition, $\text{Re}(\delta)\oplus \text{Im}(\delta)$ is a real closable derivation for all $\delta\in \Der_{1\otimes 1}(B\subset A,\tau)$. Therefore by \cite{DL92} and \cite{Sau89} (see also \cite[Chapter 4]{Cip08}) the set
    \[
        \mathcal{D}_\delta:=M\cap \dom(\overline{\text{Re}(\delta)\oplus \text{Im}(\delta)}) = M\cap \dom(\overline{\delta\oplus \delta^\dagger})
    \]
has the following properties: (i) $\mathcal{D}_\delta$ a $*$-algebra on which $\overline{\text{Re}(\delta)\oplus \text{Im}(\delta)}$ satisfies the Leibniz rule; (ii) $f(y)\in D_\delta$ for any self-adjoint $y\in \mathcal{D}_\delta$ and Lipschitz function $f\colon \R\to \R$ with $f(0)=0$; and (iii) $|y|\in \mathcal{D}_\delta$ for any $y\in \mathcal{D}_\delta$. These properties will be relevant for the proofs of (1), (2), and (3).\\

\noindent\textbf{(1):} $\mathcal{D}$ is independent of $X$ since the $\|\cdot\|_X$-norm is equivalent to $\|\cdot\|_Y$ for any other finite subset $Y=Y^*\subset A$ satisfying $B\<Y\>=A$. Also for any $p\in B\<T_X\>$ and $\delta\in \Der_{1\otimes 1}(B\subset A,\tau)$, one has
    \[
        \| \delta( p(X))\| \leq \sum_{x\in X} \| \ev_X(\partial_{x\colon B} p) \# \delta(x) \| \leq \max_{x\in X} \|\ev_X(\partial_{x\colon B} p)\| \|\delta\|_X.
    \]
This shows $A\subset \mathcal{D}$.

Now, let $\delta\in \Der_{1\otimes 1}(B\subset A,\tau)$, $y_1,y_2\in \mathcal{D}\subset \mathcal{D}_\delta$, and $\alpha\in \C$. The remarks at the beginning of the proof as well as (\ref{eqn:closures_of_real_and_imaginary_derivations}) imply
    \begin{align*}
        \| \overline{\delta}(y_1+\alpha y_2^*)\| = \| \overline{\delta}(y_1) + \alpha \overline{\delta^\dagger}(y_2)\|  &\leq (c_1+|\alpha|c_2)\|\delta\|_X\\
        \| \overline{\delta}(y_1y_2)\| = \| \overline{\delta}(y_1)\cdot y_2 + y_1\cdot \overline{\delta}(y_2)\| &\leq (c_1\|y_2\| + c_2 \|y_1\|)\|\delta\|_X,
   \end{align*}
where $c_1,c_2>0$ are independent of $\delta$. Hence $\mathcal{D}$ is a $*$-algebra.\\

\noindent\textbf{(2):} Since $f(y)-f(1)=f_0(y)$ for $f_0:= f - f(1)$ and $A\subset \mathcal{D}$ is unital, it suffices to assume $f(0)=0$. In this case, one has $f(y)\in \mathcal{D}_\delta$ for all $\delta\in \Der_{1\otimes 1}(B\subset A, \tau)$ by the remarks at the beginning of the proof. Moreover, by \cite[Equation 4.7]{Cip08} we see that
    \[
       \left(\|\overline{\text{Re}(\delta)}(f(y))\|_\ttop^2 + \|\overline{\text{Im}(\delta)}(f(y))\|_\ttop^2 \right)^{1/2} \leq \|f\|_{\text{Lip}} \left(\|\overline{\text{Re}(\delta)}(y)\|_\ttop^2 + \|\overline{\text{Im}(\delta)}(y)\|_\ttop^2 \right)^{1/2}.
    \]
Using (\ref{eqn:closures_of_real_and_imaginary_derivations}), one then has
    \begin{align*}
        \|\overline{\delta}(f(y))\|_{\ttop} &= \frac{1}{\sqrt{2}} \left(\|\overline{\delta}(f(y))\|_\ttop^2 + \|\overline{\delta^\dagger}(f(y))\|_\ttop^2 \right)^{1/2}\\
        &\leq \left(\|\overline{\text{Re}(\delta)}(f(y))\|_\ttop^2 + \|\overline{\text{Im}(\delta)}(f(y))\|_\ttop^2 \right)^{1/2}\\
        &\leq  \|f\|_{\text{Lip}} \left(\|\overline{\text{Re}(\delta)}(y)\|_\ttop^2 + \|\overline{\text{Im}(\delta)}(y)\|_\ttop^2 \right)^{1/2}\\
        &\leq \frac{1}{\sqrt{2}} \|f\|_{\text{Lip}} \left(\|\overline{\delta}(y)\|_\ttop^2 + \|\overline{\delta^\dagger}(y)\|_\ttop^2 \right)^{1/2} \leq  \|f\|_{\text{Lip}} c \|\delta\|_X,
    \end{align*}
where $c>0$ is independent of $\delta$. Thus $f(y)\in \mathcal{D}$.\\

\noindent\textbf{(3):} Consider the self-adjoint matrix
    \[
        \left( \begin{array}{cc} 0 & y^* \\ y & 0 \end{array} \right)
    \]
with absolute value
    \[
         \left( \begin{array}{cc} |y| & 0 \\ 0 & |y^*| \end{array} \right).
    \]
By the proof of the previous part, for $\delta\in \Der_{1\otimes 1}(B\subset A,\tau)$ we have
    \begin{align*}
        \left\| \left( \begin{array}{cc} \overline{\delta}(|y|) & 0 \\ 0 & \overline{\delta}(|y^*|) \end{array} \right) \right\| &\leq  \frac{1}{\sqrt{2}}\left( \left\| \left(\begin{array}{cc} 0 & \overline{\delta}(y^*) \\ \overline{\delta}(y) & 0 \end{array}\right) \right\|^2 + \left\| \left(\begin{array}{cc} 0 & \overline{\delta^\dagger}(y^*) \\ \overline{\delta^\dagger}(y) & 0 \end{array}\right) \right\|^2 \right)^{1/2}\\
        &=  \left( \|\overline{\delta}(y)\|_\ttop^2 +  \|\overline{\delta^\dagger}(y)\|_\ttop^2 \right)^{1/2} \leq \sqrt{2} c \|\delta\|_X,
    \end{align*}
where $c>0$ is independent of $\delta$. Thus $|y|\in \mathcal{D}$.\\

\noindent\textbf{(4):} The convexity of these sets is clear, and so it suffices to show they are closed in the strong operator topology. Let $(y_i)_{i\in I}\subset \mathcal{D}$ be a net converging strongly to some $y\in M$ and satisfying $\|\overline{\delta}(y_i)\|_\ttop \leq c\|\delta\|_X$ for all $\delta\in \Der_{1\otimes 1}(B\subset A, \tau)$. Consequently, $\| y - y_i\|_\tau\to 0$ and
    \[
        \limsup_{i\to \infty} \left(\| \overline{\delta}(y_i)\|_\ttop^2 + \| \overline{\delta^\dagger}(y_i)\|_\ttop^2\right)^{1/2} \leq c \left( \|\delta\|_X^2 + \|\delta^\dagger\|_X^2 \right)^{1/2}.
    \]
This shows $y\in \dom(\overline{\delta\oplus \delta^\dagger})$. Moreover
    \[
        \| \overline{\delta}(y)\|_\ttop \leq \limsup_{i\in \infty} \| \overline{\delta}(y_i)\|_\ttop \leq c \| \delta\|_X.\qedhere
    \]
\end{proof}

\begin{thm}\label{thm:fSd_of_Dirichlet_extension}
Let $(M,\tau)$ be a tracial von Neumann algebra with unital $*$-subalgebras $B\subset A\subset M$ such that $A$ is finitely generated over $B$. For $\mathcal{D}$ as in Proposition~\ref{prop:Dirichlet_extension_of_A}, $\sigma(B\subset \mathcal{D},\tau)=\sigma(B\subset A,\tau)$.
\end{thm}
\begin{proof}
By Remark~\ref{rem:starting_upper_lower_fSd_later} it suffices to show
    \[
        \sigma(B\subset A\<Y\>, \tau)= \sigma(B\subset A,\tau)
    \]
for any finite subset $Y=Y^*\subset \mathcal{D}$. Fix such a subset and consider the right $\aaop$-linear map
    \begin{align*}
        \Der_{1\otimes 1}(B\subset A\<Y\>, \tau)&\to \Der_{1\otimes 1}(B\subset A,\tau)\\
            \delta &\mapsto \delta|_A.
    \end{align*}
If $\delta|_A\equiv 0$, then $Y\subset A'' \subset \ker(\overline{\delta})$ by closability and hence $\delta\equiv 0$. Therefore the map is injective. It is surjective because for any $\delta\in \Der_{1\otimes 1}(B\subset A, \tau)$ we have $Y\subset A''\cap \dom(\overline{\delta\oplus \delta^\dagger})$ and thus $\overline{\delta}|_{A\<Y\>} \in \Der_{1\otimes 1}(B\subset A\<Y\>,\tau)$ is the preimage of $\delta$. The map is clearly bounded, and toward showing it has bounded inverse let $X=X^*\subset A$ be a finite subset satisfying $B\<X\>$. Since $Y$ is finite we have
    \[
        \| \overline{\delta}|_{A\<Y\>} \|_{X\cup Y}^2 = \| \delta\|_X^2 + \sum_{y\in Y} \|\overline{\delta}(y)\|_\ttop^2 \leq (1+|Y|c^2)\|\delta\|_X^2
    \]
for some $c>0$ independent of $\delta$. Thus the above map extends to a right $(\aaop)''$-linear homeomorphism of the closures, and hence $\sigma(B\subset A\<Y\>,\tau)=\sigma(B\subset A,\tau)$. 
\end{proof}


\bibliographystyle{amsalpha}
\bibliography{references}

\end{document}